\documentclass{article}
\usepackage{graphicx} 
\usepackage{geometry}
\usepackage{amsmath,amssymb,amsthm}
\usepackage{algorithm,algorithmic}
\usepackage{hyperref}
\usepackage{subcaption}
\usepackage[titletoc]{appendix}
\usepackage{xcolor}

\allowdisplaybreaks[4]

\newtheorem{theorem}{Theorem}[section]
\newtheorem{proposition}[theorem]{Proposition}
\newtheorem{corollary}[theorem]{Corollary}
\newtheorem{lemma}[theorem]{Lemma}
\newtheorem{remark}{Remark}

\newcommand{\tf}{\tilde{f}}
\newcommand{\R}{\mathbb{R}}
\newcommand{\rd}{\mathrm{d}}
\newcommand{\bv}{{\boldsymbol{v}}}
\newcommand{\bV}{{\boldsymbol{V}}}
\newcommand{\bx}{{\boldsymbol{x}}}
\newcommand{\bs}{{\boldsymbol{s}}}
\newcommand{\tbs}{\tilde{\boldsymbol{s}}}
\newcommand{\bT}{\boldsymbol{T}}
\newcommand{\bU}{\boldsymbol{U}}
\newcommand{\bz}{\boldsymbol{z}}
\newcommand{\bu}{\boldsymbol{u}}
\newcommand{\bbT}{\mathbb{T}}

\title{A score-based particle method for homogeneous Landau equation \footnote{ This work is support in part by NSF grant DMS-1846854, UMN DSI-SSG-4886888864, and The Simons Fellowship.}}
\author{Yan Huang \thanks{School of Mathematics, University of Minnesota Twin Cities, Minneapolis, MN 55455, USA. (huan2728@umn.edu)} and Li Wang \thanks{School of Mathematics, University of Minnesota Twin Cities, Minneapolis, MN 55455, USA. (liwang@umn.edu)} }

\date{}

\begin{document}

\maketitle

\begin{abstract}
We propose a novel score-based particle method for solving the Landau equation in plasmas, that seamlessly integrates learning with structure-preserving particle methods \cite{carrillo2020particle}. Building upon the Lagrangian viewpoint of the Landau equation, a central challenge stems from the nonlinear dependence of the velocity field on the density.  Our primary innovation lies in recognizing that this nonlinearity is in the form of the score function, which can be approximated dynamically via techniques from score-matching. The resulting method inherits the conservation properties of the deterministic particle method while sidestepping the necessity for kernel density estimation in \cite{carrillo2020particle}. This streamlines computation and enhances scalability with dimensionality. Furthermore, we provide a theoretical estimate by demonstrating that the KL divergence between our approximation and the true solution can be effectively controlled by the score-matching loss. Additionally, by adopting the flow map viewpoint, we derive an update formula for exact density computation. Extensive examples have been provided to show the efficiency of the method, including a physically relevant case of Coulomb interaction.
\end{abstract}

\section{Introduction}
The Landau equation stands as one of the fundamental kinetic equations, modeling the evolution of charged particles undergoing Coulomb interaction \cite{landau1958kinetic}. It is particularly useful for plasmas where collision effects become non-negligible. Computing the Landau equation presents numerous challenges inherent in kinetic equations, including high dimensionality, multiple scales, and strong nonlinearity and non-locality. On the other hand, deep learning has progressively transformed the numerical computation of partial differential equations by leveraging neural networks' ability to approximate complex functions and the powerful optimization toolbox. However, straightforward application of deep learning to compute PDEs often encounters training difficulties and leads to a loss of physical fidelity. In this paper, we propose a score-based particle method that elegantly combines learning with structure-preserving particle methods. This method inherits the favorable conservative properties of deterministic particle methods while relying only on light training to dynamically obtain the score function over time. The learning component replaces the expensive density estimation in previous particle methods, drastically accelerating computation.

In general terms, the Landau equation takes the form 
\begin{equation*} 
    \partial_t \tf_t + \bv \cdot \nabla_\bx \tf_t + E \cdot \nabla_\bv \tf_t = Q(\tf_t, \tf_t) := \nabla_\bv \cdot \left[\int_{\R^d} A(\bv-\bv_*) \left( \tf_t(\bv_*) \nabla_\bv \tf_t(\bv) - \tf_t(\bv) \nabla_{\bv_*} \tf_t(\bv_*) \right) \rd \bv_*\right] \,,
\end{equation*}
where $\tf_t(\bx,\bv)$ for $(\bx,\bv) \in \Omega \times \R^d$ with $d\geq 2$, is the mass distribution function of charged particles, such as electrons and ions. 
$E$ represents the electric field, which can be prescribed or obtained self-consistently through the Poisson or Ampere equation. The collision kernel $A$ is expressed as
\begin{equation} \label{A}
    A(\bz)=C_{\gamma}|\bz|^{\gamma} (|\bz|^{2}I_d - \bz \otimes \bz)=:C_{\gamma}|\bz|^{\gamma+2} \Pi(\bz) \,,   
\end{equation}
where $C_{\gamma} > 0$ is the collision strength, and $I_d$ is the identity matrix. Consequently, $\Pi(\bz)$ denotes the projection into the orthogonal complement of $\{\bz\}$. The parameter $\gamma$ can take values within the range $-d-1 \leq \gamma \leq 1$. Among them, the most interesting case is when $d = 3$ and $\gamma = -3$, which corresponds to the Coulomb interaction in plasma \cite{degond1992fokker, Villani1998b}. Alternatively, the case $\gamma = 0$ is commonly referred to as the Maxwell case. In this scenario, the equation reduces to a form of degenerate linear Fokker-Planck equation, preserving the same moments as the Landau equation \cite{villani1998a}.

In this paper, our primary focus is on computing the collision operator, and therefore we shall exclusively consider the spatially homogeneous Landau equation
\begin{equation}\label{Landau_Homogeneous}
    \partial_t \tf_t = Q(\tf_t, \tf_t) \,.
\end{equation}
Key properties of $Q$ include conservation and entropy dissipation, which can be best understood formally through the following reformulation:
\begin{equation*}
    Q(\tf_t, \tf_t) = \nabla \cdot \left[ \int_{\R^d} A(\bv-\bv_*) \tf_t \tf_{t,*} (\nabla\log \tf_t - \nabla_* \log \tf_{t,*}) \rd\bv_* \right] \,,
\end{equation*}
where we have used the abbreviated notation
\begin{equation*}
    \tf_t:=\tf_t(\bv), ~ 
    \tf_{t,*}:=\tf_t(\bv_*), ~
    \nabla := \nabla_{\bv}, ~
    \nabla_* := \nabla_{\bv_*} \,.
\end{equation*}
For an appropriate test function $\phi = \phi(\bv)$, it admits the weak formulation:
\begin{equation*}
    \int_{\R^d} Q(\tf_t, \tf_t)\phi \rd\bv = -\frac{1}{2} \iint_{\R^{2d}} (\nabla\phi - \nabla_*\phi_*) \cdot A(\bv-\bv_*) (\nabla\log \tf_t - \nabla_*\log \tf_{t,*}) \tf_t \tf_{t,*} \rd\bv \rd\bv_* \,.
\end{equation*}
Then, choosing $\phi(\bv) = 1, \bv, |\bv|^2$ leads to the conservation of mass, momentum, and energy. Inserting $\phi(\bv) = \log \tf_t(\bv)$, one
obtains the formal entropy decay with dissipation given by
\begin{equation}\label{entropy_decay_rate}
    \begin{split}
        \frac{\rd}{\rd t} \int_{\R^d} \tf_t \log \tf_t \rd\bv 
        & = \int_{\R^d} Q(\tf_t, \tf_t) \log \tf_t \rd\bv \\
        & = -\frac{1}{2} \iint_{\R^{2d}} B_{\bv, \bv_*} \cdot A(\bv-\bv_*) B_{\bv, \bv_*} \tf_t \tf_{t,*} \rd\bv \rd\bv_* \leq 0 \,,
    \end{split}
\end{equation}
where we have used the fact that $A$ is symmetric and semi-positive definite, and $B_{\bv, \bv_*}:=\nabla\log \tf_t - \nabla_*\log \tf_{t,*}$. The equilibrium distribution is given by the Maxwellian
\begin{equation*}
    M_{\rho, \bu, T} = \frac{\rho}{(2\pi T)^{d/2}} \exp\left(-\frac{|\bv-\boldsymbol{u}|^2}{2T}\right) \,,
\end{equation*}
for $\rho, T$ and $\bu$ given by
\begin{equation*}
    \rho = \int_{\R^d} \tf_t(\bv) \rd\bv, \quad \rho\bu = \int_{\R^d} \bv \tf_t(\bv) \rd\bv, \quad \rho |\bu|^2 + \rho dT = \int_{\R^d} |\bv|^2 \tf_t(\bv) \rd\bv \,.
\end{equation*}
More rigorous approach can be found in \cite{Villani1998b, gualdani2017spectral}.

Theoretical understanding of the well-posedness of the homogeneous equation \eqref{Landau_Homogeneous} with hard potential ($\gamma>0$) is now well-established, primarily due to the seminal works in \cite{villani2000a, villani2000b} and related literature. The regularity of the solution, such as moment propagation, has also been rigorously established. A pivotal aspect involves leveraging finite entropy dissipation, leading to the robust notion of `H-solution' introduced by Villani \cite{Villani1998b}.
Less is currently known about soft potentials ($\gamma<0$). One significant advancement in this regard was a global existence and uniqueness result by Guo \cite{Guo2002}, in which a sufficient regular solution close to the Maxwellian is considered in the spatially inhomogeneous case. A novel interpretation of the homogeneous Landau equation as a gradient flow has emerged \cite{carrillo2022landau}, along with its connection to the Boltzmann equation via the gradient flow perspective \cite{carrillo2022grazing}. A recent breakthrough on the global well-posedness of the Landau equation is made in \cite{guillen2023landau}.

Various numerical methods have been developed to compute the Landau operator, including the Fourier-Galerkin spectral method \cite{spectrallandau}, the direct simulation Monte Carlo method \cite{dimarco2010direct, ROSIN2014140}, the finite difference entropy method \cite{degond1994entropy, buet1998conservative}, and the deterministic particle method \cite{carrillo2020particle}. Among these, we are particularly interested in the deterministic particle method \cite{carrillo2020particle}, which preserves all desirable physical quantities, including conservation of mass, momentum, energy, and entropy dissipation. 

The main idea in \cite{carrillo2020particle} is to reformulate the homogeneous Landau equation \eqref{Landau_Homogeneous} into a continuity equation with a gradient flow structure:
\begin{equation}\label{Landau_gf}
    \begin{split}
        \partial_t \tf_t &= -\nabla \cdot (\tf_t \bU[\tf_t]) \,, \\
        \bU[\tf_t] &= -\int_{\R^d} A(\bv-\bv_*) \left(\nabla\frac{\delta \mathcal H}{\delta \tf_t} - \nabla_*\frac{\delta \mathcal H_*}{\delta \tf_{t,*}}\right) \tf_{t,*}\rd \bv_* \,,~~ \mathcal H = \int_{\R^d} \tf_t \log \tf_t \rd \bv \,,
    \end{split}
\end{equation}
where $\bU[\tf_t]$ is the velocity field. Employing a particle representation
\begin{equation*}
    \tf_t^N(\bv) = \sum_{i = 1}^N w_i \delta (\bv -\bv_i(t)) \,,
\end{equation*}
where $N$ is the total number of particles, $\bv_i(t)$ and $w_i$ are the velocity and weight of the particle $i$, respectively. Here, $\delta $ is the Dirac-Delta function. Subsequently, the particle velocities update following the characteristics of \eqref{Landau_gf}:
\begin{align} \label{p1}
    \dot{\bv}_i(t) = \bU[\tf_t^N](\bv_i(t), t) = 
    -\sum_{j=1}^N w_j A(\bv_i - \bv_j) \left(\nabla\frac{\delta \mathcal H^N}{\delta \tf_t} (\bv_i) - \nabla\frac{\delta \mathcal H^N}{\delta \tf_t} (\bv_j) \right) \,.
\end{align}
To make sense of the entropy $\mathcal{H}^N$, a crucial aspect of this method involves replacing $\mathcal{H}^N$ with a regularized version:
\begin{align} \label{p2}
    \mathcal{H}^N_\varepsilon := \int \tf^N_{t,\varepsilon} \log \tf^N_{t,\varepsilon} \rd\bv \,, ~~
    \tf^N_{t,\varepsilon} := \psi_\varepsilon \ast \tf_t^N = \sum_{i=1}^N w_i \psi_\varepsilon(\bv - \bv_i(t)) \,,
\end{align}
where $\psi_\varepsilon$ is a mollifier such as a Gaussian. This way of regularization follows the previous work on the blob method for diffusion \cite{carrillo2019blob}. The convergence of this method is obtained in \cite{convparticle}, and a random batch version is also available in \cite{randombatch}.

While being structure-preserving, a significant bottleneck in (\ref{p1}--\ref{p2}) lies in the necessity to compute $\tf^N_{t,\varepsilon}$ in \eqref{p2}, a task often referred to as kernel density estimation. This task is further compounded when computing the velocity via $\nabla \frac{\delta \mathcal H^N_\varepsilon}{\delta \tf_t}$ (see \eqref{double}). It is widely recognized that this estimation scales poorly with dimension. To address this challenge, the main concept in this paper is to recognize that the nonlinear term in the velocity field, which depends on density, serves as the score function, i.e., 
\begin{align*}
    \bU[\tf_t] = -\int_{\R^d} A(\bv-\bv_*) (\underbrace{\nabla\log \tf_t}_{\text{score}} - \nabla_*\log \tf_{t,*}) \tf_{t,*} \rd \bv_* \,,
\end{align*}
and it can be efficiently learned from particle data by score-matching trick \cite{hyvarinen05a, Vincent11}. More particularly, start with a set of $N$ particles $\{\bV_i\}_{i=1}^N \stackrel{i.i.d.}\sim \tf_0 = f_0$, 
we propose the following update process for their later time $t$ velocities $ \{\bv_i(t)\}_{i=1}^N$ as
\begin{equation*}
    \begin{cases}{}
     \bs (\bv) \in \arg \min \frac{1}{N} \sum_{i=1}^N |\bs(\bv_i(t))|^2 + 2 \nabla \cdot \bs(\bv_i(t)) \,,
     \\ \\ \dot{\bv}_i(t) =- \frac1N \sum_{j=1}^N A(\bv_i(t)-\bv_j(t)) (\bs(\bv_i(t)) - \bs(\bv_j(t))) \,.
    \end{cases}
\end{equation*}
This process involves initially learning a score function at each time step using the current particle velocity information and subsequently utilizing the learned score to update the particle velocities. Closest to our approach is the score-based method for the Fokker-Planck equation \cite{boffi2206probability,lu2024score}, where a similar idea of dynamically learning the score function is employed. Additionally, learning velocities instead of solutions, such as self-consistent velocity matching for Fokker-Planck-type equations \cite{shen22a, shen2023entropydissipation, velmatch}, and the DeepJKO method for general gradient flows \cite{lee2023deep}, are also related. However, a key distinction lies in our treatment of the real Landau operator, which is significantly more challenging than the Fokker-Planck operator. In this regard, our work represents the initial attempt at leveraging the score-matching technique to compute the Landau operator. Another related work is \cite{yifei2024}, where the author approximates the Wasserstein gradient direction using a class of two-layer neural networks with ReLU activation, and proposes a semi-definite program relaxation to find such an approximation.

It's worth noting that compared to other popular neural network-based PDE solvers, such as the physics-informed neural network (PINN) \cite{pinn}, the deep Galerkin method (DGM) \cite{dgm}, the deep Ritz Method \cite{dritz}, and the weak adversarial network (WAN) \cite{wan}, as well as those specifically designed for kinetic equations \cite{Jin2023, jin2024mulki, Jin2024vpfp, Lu2022}, the proposed method requires very {\it light training}. The sole task is to sequentially learn the score, and considering that the score doesn't change significantly over each time step, only a small number of training operations (approximately 25 iterations) are needed. This method offers a compelling combination of classical methods with complementary machine learning techniques.

The rest of the paper is organized as follows. In the next section, we introduce the main formulation of our method based on a flow map reformulation and present the score-based algorithm. In Section \ref{sec:3}, we establish a stability estimate using relative entropy, theoretically justifying the controllability of our numerical error by score-matching loss. Section \ref{sec:4} provides an exact update formula for computing the density along the trajectories of the particles. Numerical tests are presented in Section \ref{sec:5}, and the paper is concluded in Section \ref{sec:6}.

\section{A score-based particle method}
This section is devoted to the derivation of a deterministic particle method based on score-matching. As mentioned in the previous section, our starting point is the reformulation of the homogeneous Landau equation \eqref{Landau_Homogeneous} into a continuity equation:
\begin{equation}\label{Landau_Continuity}
    \begin{split}
        \partial_t \tf_t &= -\nabla \cdot (\tf_t \bU[\tf_t]) \,, \\
        \bU[\tf_t] &= -\int_{\R^d} A(\bv-\bv_*) (\nabla\log \tf_t - \nabla_*\log \tf_{t,*}) \tf_{t,*} \rd \bv_* \,,
    \end{split}
\end{equation}
where $\bU[\tf_t]$ is the velocity field.

\subsection{Lagrangian formulation via transport map}
The formulation \eqref{Landau_Continuity} gives access to the Lagrangian formulation. In particular, let $\bT(\cdot, t)$ be the flow map, then for a particle started with velocity $\bV$, its later time velocity can be obtained as $\bv(t):=\bT(\bV, t)$, which satisfies the following ODE: 
\begin{equation}\label{flow_map0}
    \frac{\rd}{\rd t}\bT(\bV,t) 
    = -\int_{\R^d} A(\bT(\bV, t)-\bv_*) \left[\nabla \log \tf_t (\bT(\bV, t)) - \nabla_* \log \tf_t(\bv_*)\right] \tf_t(\bv_*) \rd \bv_* \,.
\end{equation}
Using the fact that the solution to \eqref{Landau_Continuity} can be viewed as the pushforward measure under the flow map, i.e., 
\begin{equation*}
    \tf_t(\cdot) = \bT(\cdot, t)_{\#} f_0(\cdot) \,,
\end{equation*}
we can rewrite \eqref{flow_map0} as
\begin{equation}\label{flow_map}
\begin{cases}{}
    \frac{\rd}{\rd t}\bT(\bV,t) = -\int_{\R^d} A(\bT(\bV, t)-\bT(\bV_*, t)) \left[\nabla \log \tf_t (\bT(\bV, t)) - {\nabla} \log \tf_t(\bT(\bV_*, t))\right] f_0(\bV_*) \rd \bV_* \,, \\ 
    \bT(\bV, 0) = \bV \,.
\end{cases}
\end{equation}
Therefore, if we start with a set of $N$ particles $\{\bV_i\}_{i=1}^N \stackrel{i.i.d.}\sim f_0$, then the later velocity $\bv_i(t):= \bT(\bV_i, t)$ satisfies 
\begin{equation}\label{particle0}
\begin{cases}{}
    \dot{\bv}_i(t) = -\frac{1}{N}\sum_{j=1}^N A(\bv_i(t)-\bv_j(t)) \left[ \nabla \log \tf_t (\bv_i(t)) - \nabla \log \tf_t(\bv_j(t)) \right] \,, \\ 
    \bv_i(0) = \bV_i \,,
\end{cases}
\end{equation}
for $1\leq i\leq N$.  

This formulation immediately has the following favorable properties. Hereafter,  we denote $\tilde \bs_t := \nabla \log \tilde f_t$. 

\begin{proposition} 
    The particle formulation \eqref{particle0} conserves mass, momentum, and energy. Specifically, the following quantities remain constant over time:
    \begin{equation*}
        \sum_{i=1}^N \frac{1}{N}, \quad \frac{1}{N} \sum_{i=1}^N \bv_i(t), \quad \frac{1}{N} \sum_{i=1}^N |\bv_i(t)|^2 \,.
    \end{equation*}
\end{proposition}
\begin{proof}
    Indeed, for $\phi(\bv) = 1, \bv, |\bv|^2$, we have
    \begin{equation*}
        \begin{split}
            \frac{\rd}{\rd t} \frac{1}{N} \sum_{i=1}^N \phi(\bv_i(t)) 
            & = \frac{1}{N} \sum_{i=1}^N \nabla\phi(\bv_i(t)) \cdot \dot{\bv}_i(t) \\
            & = -\frac{1}{N^2} \sum_{i,j=1}^N \nabla\phi(\bv_i(t)) \cdot A(\bv_i(t)-\bv_j(t)) [\tbs_t(\bv_i(t)) - \tbs_t(\bv_j(t))] \\
            & = -\frac{1}{2N^2} \sum_{i,j=1}^N [\nabla\phi(\bv_i(t)) - \nabla\phi(\bv_j(t))] \cdot A(\bv_i(t)-\bv_j(t)) [\tbs_t(\bv_i(t)) - \tbs_t(\bv_j(t))] = 0 \,,
        \end{split}
    \end{equation*}
    which leads to the conversation of mass, momentum, and energy. 
\end{proof}

\subsection{Learning the score}
Implementing \eqref{particle0} faces a challenge in representing the score function $\tbs_t = \nabla \log \tf_t$ using particles. A natural approach is to employ kernel density estimation:
\[
\tf_t(\bv_i) \approx  \frac{1}{N} \sum_{j=1}^N \psi_\varepsilon(\bv_i - \bv_j(t)) \,,
\]
where $\psi_\varepsilon$ could be, for instance, a Gaussian kernel. However, this estimation becomes impractical with increasing dimensions due to scalability issues. Instead, we propose utilizing the score-matching technique to directly learn the score $\tbs_t$ from the data. This technique dates back to \cite{hyvarinen05a, Vincent11} and has flourished in the context of score-based generative modeling \cite{song2020generative, song2020techscorebased, song2021scorebased} recently, and has been used to compute the Fokker-Planck type equations \cite{boffi2206probability, lu2024score}. 

Now let $\bs_t(\bv): \R^d \to \R^d$ be an approximation to the exact score $\tbs_t(\bv)$, we define the score-based Landau equation as
\begin{equation}\label{Landau_Score}
    \partial_t f_t = -\nabla \cdot (\bU^\delta f_t) \,, ~
    \bU^\delta := -\int_{\R^d} A(\bv-\bv_*)(\bs_t(\bv) - \bs_t(\bv_*))f_{{t,*}} \rd \bv_* \,.
\end{equation}
 Let $\bT(\cdot,t)$ be the corresponding flow map of \eqref{Landau_Score}, then we have
\begin{flalign*}
        \frac{\rd}{\rd t}\bT(\bV,t) = -\int_{\mathbb{R}^{d}} A(\bT(\bV,t)-\bT(\bV_*,t)) \left[\bs_t(\bT(\bV,t)) - \bs_{t}(\bT(\bV_*,t)) \right] f_0(\bV_*) \rd\bV_*, \ \bT(\bV,0)=\bV \,.
\end{flalign*}
For any $t\geq 0$, to get $\tbs_t$, we seek to minimize 
\begin{equation}
        \min_{\bs_t} \int_{\R^d} \left[|\bs_t(\bT(\bV,t))|^2 + 2 \nabla \cdot \bs_t(\bT(\bV,t))\right] f_0(\bV) \rd\bV \,.
\end{equation}

In practice,  given $N$ initial particles with velocities $\{\bV_i\}_{i=1}^N \stackrel{i.i.d.}\sim f_0$, at each time $t$, we train the neural network $\bs_t$ to minimize the implicit score-matching loss, i.e., 
\begin{equation*}
    \min_{\bs_t}\frac{1}{N} \sum_{i=1}^N |\bs_t(\bT(\bV_i,t))|^2 + 2 \nabla \cdot \bs_t(\bT(\bV_i,t)) = \frac{1}{N} \sum_{i=1}^N |\bs_t(\bv_i(t))|^2 + 2 \nabla \cdot \bs_t(\bv_i(t)) \,,
\end{equation*}
and then evolves the particles via \eqref{particle0} with learned $\bs_t$, i.e., replacing $\nabla \log\tf_t(\bv)$ by $\bs_t(\bv)$.

\subsection{Algorithm}
We hereby summarize a few implementation details. The time discretization of \eqref{particle0} is done by the forward Euler method. The initial neural network $\bs_0$ is trained to minimize the relative loss compared to the analytical form:
\begin{equation}\label{L2_loss}
    \ell_1 := \frac{\int_{\R^d} | \bs_0(\bv) - \nabla \log f_0(\bv) |^2 f_0(\bv) \rd\bv}{\int_{\R^d} | \nabla\log f_0(\bv) |^2 f_0(\bv) \rd\bv} 
    \approx \frac{\sum_{i=1}^N | \bs_0(\bV_i) - \nabla \log f_0(\bV_i) |^2}{\sum_{i=1}^N | \nabla\log f_0(\bV_i) |^2} \,.
\end{equation}
For the subsequent steps, we initialize $\bs_n$ using the previously trained $\bs_{n-1}$ and train it to minimize the implicit score-matching loss at time $t_n$,
\begin{equation}\label{ISM_loss}
    \ell_2 := \frac{1}{N} \sum_{i=1}^N |\bs_n(\bv_i^n)|^2 + 2 \nabla \cdot \bs_n(\bv_i^n), \quad \bv_i^n \approx \bv_i(t_n) \,.
\end{equation}
Note that to avoid the expensive computation of the divergence, especially in high dimensions, the denoising score-matching loss function introduced in \cite{Vincent11} is often utilized. However, here we still compute the divergence exactly through automatic differentiation, as it allows for precise tracking of the optimization process's convergence.  
Once the score is learned from \eqref{ISM_loss}, the velocity of particles can be updated via
\begin{align}\label{vup}
    \bv_i^{n+1} = \bv_i^n - \Delta t \frac1N \sum_{j=1}^N A(\bv_i^n-\bv_j^n) (\bs_n(\bv_i^n) - \bs_n(\bv_j^n)) \,.
\end{align}

The procedure of the score-based particle method is summarized in Algorithm \ref{algorithm_no_densitiy}.
\begin{algorithm}[H]
\caption{Score-based particle method for the homogeneous Landau equation}
\label{algorithm_no_densitiy}
    \begin{algorithmic}[1]
        \REQUIRE $N$ initial particles $\{\bV_i\}_{i=1}^N \stackrel{i.i.d.}\sim f_0$; time step $\Delta t$ and the total number of time steps $N_T$; error tolerance $\delta$ for the initial score-matching; max iteration number $I_{\max}$ for the implicit score-matching.
        \ENSURE Score neural networks $\bs_{n-1}$ and particles $\{\bv_i^n\}_{i=1}^N$ for all $n=1, \cdots, N_T$.
        
        \STATE Initialize neural network $\bs_0$.
        \WHILE{$\ell_1 > \delta$}
            \STATE Update parameters of $\bs_0$ by applying optimizer to $\ell_1$ \eqref{L2_loss}.
        \ENDWHILE
        \FOR{$i=1,\cdots,N$}
            \STATE $\bv_i^{1} = \bV_i - \Delta t \frac1N \sum_{j=1}^N A(\bV_i-\bV_j) (\bs_0(\bV_i)-\bs_0(\bV_j))$.
        \ENDFOR
        
        \FOR{$n=1,\cdots, N_T-1$}
            \STATE Initialize neural network $\bs_n$ using the previously trained $\bs_{n-1}$.
            \STATE Set $I=0$.
            \WHILE{$I < I_{\max}$} 
                \STATE Update parameters of $\bs_n$ by applying optimizer to $\ell_2$ \eqref{ISM_loss}.
                \STATE $I=I+1$.
            \ENDWHILE
            \FOR{$i=1,\cdots,N$}
                \STATE obtain $\bv_i^{n+1}$ from $\bv_i^{n}$ via \eqref{vup}.
            \ENDFOR
        \ENDFOR
    \end{algorithmic}
\end{algorithm}

We would like to note that the computational complexity of the score learning step (line 12) in Algorithm~\ref{algorithm_no_densitiy} is $\mathcal {O}(N)$. See also Fig.~\ref{fig_eg4_time} for numerical evidence. However, the particle update step (line 16) remains $\mathcal{O}(N^2)$. This cost can be reduced using the random batch method, as explored in \cite{randombatch}, which accelerates the approach presented in \cite{carrillo2020particle}.

Several macroscopic quantities can be computed using particles at  time $t_n$, including mass, momentum, and energy:
\begin{equation*}
    \sum_{i=1}^N \frac{1}{N} \,, \quad \frac{1}{N} \sum_{i=1}^N \bv_i^n \,, \quad \frac{1}{N} \sum_{i=1}^N |\bv_i^n|^2 \,,
\end{equation*}
and the estimated entropy decay rate:
\begin{equation*} 
    -\frac{1}{N^2} \sum_{i,j=1}^N \bs_n(\bv_i^n) \cdot A(\bv_i^n - \bv_j^n) (\bs_n(\bv_i^n) - \bs_n(\bv_j^n)) \,.
\end{equation*}

\begin{proposition}\label{prop:2.2}
    The score-based particle method conserves mass and momentum exactly, while the energy is conserved up to $\mathcal{O}(\Delta t)$. 
\end{proposition}
\begin{proof}
    Note that mass is trivially conserved. To see the momentum conservation, observe that
    \begin{equation*}
        \bv_i^{n+1} = \bv_i^n - \Delta t \underbrace{\frac1N \sum_{j=1}^N A(\bv_i^n-\bv_j^n) (\bs_n(\bv_i^n) - \bs_n(\bv_j^n))}_{G(\bv_i^n)} \,.
    \end{equation*}
    Multiplying both sides by $\frac{1}{N}$ and sum over $i$, we obtain
    \begin{equation*}
    \begin{split}
        \frac{1}{N} \sum_{i=1}^N \bv_i^{n+1} - \frac{1}{N} \sum_{i=1}^N \bv_i^n
        & = -\Delta t \frac{1}{N^2} \sum_{i,j=1}^N A(\bv_i^n-\bv_j^n) (\bs_n(\bv_i^n) - \bs_n(\bv_j^n)) \\
        & = \Delta t \frac{1}{N^2} \sum_{i,j=1}^N A(\bv_i^n-\bv_j^n) (\bs_n(\bv_i^n) - \bs_n(\bv_j^n)) = 0 \,.
    \end{split}
    \end{equation*}
   Here, the second equality is obtained by switching $i$ and $j$ and using the symmetry of matrix $A$. 
   
   For the energy,
  note that 
   \begin{equation*}
       \begin{split}
           \sum_{i=1}^N \bv_i^n \cdot G(\bv_i^n) 
           & = \frac{1}{N} \sum_{i,j=1}^N \bv_i^n \cdot  A(\bv_i^n-\bv_j^n) (\bs_n(\bv_i^n) - \bs_n(\bv_j^n)) \\
           & = \frac{1}{2N} \sum_{i,j=1}^N (\bv_i^n-\bv_j^n) \cdot A(\bv_i^n-\bv_j^n) (\bs_n(\bv_i^n) - \bs_n(\bv_j^n)) = 0 \,,
       \end{split}
   \end{equation*}
   where we use the projection property of $A$. Thus
   \begin{equation*}
       \begin{split}
           \frac{1}{N} \sum_{i=1}^N |\bv_i^{n+1}|^2 
           & = \frac{1}{N} \sum_{i=1}^N |\bv_i^n - \Delta t G(\bv_i^n)|^2 \\
           & = \frac{1}{N} \sum_{i=1}^N |\bv_i^n|^2 - 2 \Delta t \frac{1}{N} \sum_{i=1}^N \bv_i^n \cdot G(\bv_i^n) + \Delta t^2 \frac{1}{N} \sum_{i=1}^N |G(\bv_i^n)|^2  \\
           & = \frac{1}{N} \sum_{i=1}^N |\bv_i^n|^2 + \Delta t^2 \frac{1}{N} \sum_{i=1}^N |G(\bv_i^n)|^2 \,,
       \end{split}
   \end{equation*}
   which implies energy is conserved up to $\mathcal{O}(\Delta t)$.
\end{proof}

\begin{remark}
    It is possible to achieve exact momentum and energy conservation using the midpoint discretization in time \cite{Hirvijoki_2021}: 
    \begin{equation*}
        \bv_i^{n+1} = \bv_i^n - \Delta t \frac1N \sum_{j=1}^N A(\bv_i^{n+\frac12}-\bv_j^{n+\frac12}) (\bs_i - \bs_j) \,, \text{ where } \bv_i^{n+\frac12} = \frac{\bv_i^n + \bv_i^{n+1}}{2} \,.
    \end{equation*}
    Here $\bs_i = \bs_n(\bv_i^n)$ or $\bs_i = \bs(\bv_i^{n+\frac12})$. The proof of the momentum conservation is the same as Proposition \ref{prop:2.2}. The energy is conserved since: 
    \begin{equation*}
        \begin{split}
            \frac{1}{N} \sum_{i=1}^N (|\bv_i^{n+1}|^2 - |\bv_i^n|^2) 
            & = \frac{1}{N} \sum_{i=1}^N (\bv_i^{n+1} - \bv_i^n) \cdot (\bv_i^{n+1} + \bv_i^n) \\
            & = - \frac{2 \Delta t}{N^2} \sum_{i=1}^N \sum_{j=1}^N (\bs_i - \bs_j) A(\bv_i^{n+\frac12}-\bv_j^{n+\frac12}) \bv_i^{n+\frac12} \\
            & = - \frac{\Delta t}{N^2} \sum_{i=1}^N \sum_{j=1}^N (\bs_i - \bs_j) A(\bv_i^{n+\frac12}-\bv_j^{n+\frac12}) (\bv_i^{n+\frac12} - \bv_j^{n+\frac12}) = 0 \,.
        \end{split}
    \end{equation*}
The tradeoff is that this method requires an implicit update for $\bv^{n+1}_i$, 
which necessitates a fixed-point iteration. This will be explored in future work.
\end{remark}

\section{Theoretical analysis}\label{sec:3}
In this section, we provide a theoretical justification for our score-based formulation. In particular, we show that the KL divergence between the computed density obtained from the learned score and the true solution can be controlled by the score-matching loss. This result is in the same vein as Proposition 1 in \cite{boffi2206probability} and Theorem 1 in \cite{lu2024score}, but with significantly more details due to the intricate nature of the Landau operator. A similar relative entropy approach for obtaining quantitative propagation of chaos-type results for the Landau-type equation has also been recently established in \cite{carrillo2024mean}.

To simplify the analysis, we assume that $\bv$ is on the torus $\bbT^d$. This is a common setting,  as the universal function approximation of neural networks is typically applicable only over a compact set. Additionally, in this setting, the boundary terms resulting from integration by parts vanish. We then make the following additional assumptions:

\begin{itemize}
    \item[] (A.1)  The collision kernel $A$ satisfies $\lambda_1 I_d \preceq A \preceq \lambda_2 I_d$, $0< \lambda_1 \leq \lambda_2$. This allows us to avoid the potential degeneracy and singularity at the origin.
    \item[] (A.2) Assume that the initial condition $f_0$ satisfies $f_0 \geq 0$, $f_0 \in W^{2,\infty}(\bbT^d)$, and $\int_{\bbT^d} f_0 \rd\bv=1$. This guarantees that the solution of \eqref{Landau_Continuity} satisfies $\tf_t \in L^{\infty}([0,T], W^{2,\infty}(\bbT^d))$, $\tf_t \geq 0$, and $\int_{\bbT^d} \tf_t \rd\bv=1$ for all $t \in[0,T]$. In addition, we assume the solution of \eqref{Landau_Score} satisfies $f_t \in L^{\infty}([0,T], W^{2,\infty}(\bbT^d))$, $f_t \geq 0$, and $\int_{\bbT^d} f_t \rd\bv=1$ for all $t \in[0,T]$. 

    \item[] (A.3) The solution $\tf_t$ to the original Landau equation \eqref{Landau_Continuity} satisfies 
    \begin{equation*}
        \sup_{(v, t) \in \bbT^d \times [0,T]} |\nabla\log\tf_t| \leq M 
    \end{equation*}
    for some constant $M$. 
\end{itemize}

Note that assumption (A.2) can be satisfied under assumption (A.1), following the classical theory of advection-diffusion equations. Assumption (A.3) is a direct consequence of (A.2). We list it here solely for ease of later reference. Regarding assumption (A.1), it is primarily needed to estimate the term $I_1$ (see \eqref{I1}) in the main theorem. However, its necessity can be relaxed to:
\[
\int_{\bbT^d} A(\bv - \bv_*) f(\bv) \rd \bv \succeq \lambda_1 I_d
\]
for any probability measure $f(\bv) \rd \bv$, and  $\lambda_1 >0$. 
To this end, we present the following proposition to justify this assumption partially.

\begin{proposition}\label{prop:pd}
    For any probability measure $\rho(\bv) \rd\bv$ on the torus $\bbT^d$, $d=2, 3$, we have
    $$\int_{\bbT^3} A(\bv-\bv_*) \rho(\bv_*) \rd\bv_* \succ 0 \,.$$
\end{proposition}
\begin{proof}
    Since the integration domain is a torus, it suffices to show that $\int_{\bbT^2} A(\bv) \rho(\bv) \rd\bv \succ 0$. Denote 
    \[
     \xi_1=\int_{\bbT^2} \rho v_1^2 |\bv|^{\gamma} \rd\bv \,, \quad 
     \xi_2=\int_{\bbT^2} \rho v_1v_2 |\bv|^{\gamma} \rd\bv \,, \quad 
    \xi_3=\int_{\bbT^2} \rho v_2^2 |\bv|^{\gamma} \rd\bv \,,
    \]
    then we have
    \begin{equation*}
        \int_{\bbT^2} A(\bv) \rho(\bv) \rd\bv = \int_{\bbT^2} |\bv|^{\gamma}
        \begin{bmatrix}
            v_2^2 & -v_1 v_2 \\
            -v_1 v_2 & v_1^2
        \end{bmatrix}
        \rho \rd\bv 
        = \begin{bmatrix}
            \xi_3 & -\xi_2 \\
            -\xi_2 & \xi_1
        \end{bmatrix}
        =: B \,.
    \end{equation*}
    The eigenvalues $\lambda$ of $B$ are given by
    \begin{equation*}
        \begin{split}
            (\lambda-\xi_3)(\lambda-\xi_1) - \xi_2^2=0 
            & \implies \lambda^2 - (\xi_1+\xi_3)\lambda + \xi_1\xi_3 -\xi_2^2=0 \,. \\
            & \implies \lambda = \frac{(\xi_1+\xi_3) \pm \sqrt{(\xi_1+\xi_3)^2 - 4(\xi_1\xi_3 -\xi_2^2)}}{2} \,.
        \end{split}
    \end{equation*}
    Denote $\xi=\int_{\bbT^2} \rho |\bv|^{\gamma+2} d\bv$. Note that $\xi_1 + \xi_3=\xi$ and $\xi_1\xi_3 > \xi_2^2$ by Cauchy-Schwarz inequality. Moreover, the eigenvalues of $B$ are real since $B$ is symmetric. Thus we have
    \begin{equation*}
        \lambda = \frac{\xi \pm \sqrt{\xi^2 - 4(\xi_1\xi_3 -\xi_2^2)}}{2} > 0 \,,
    \end{equation*}
    which proves the positive-definiteness. The proof for $d=3$ is given in Appendix \ref{proof:pd}.
\end{proof}

We now state the main theorem in this section. 
\begin{theorem}[Time evolution of the KL divergence on $\bbT^d$]\label{KL_evolution}
    Let $\tilde{f}_t$ and $f_t$ denote the solutions to the homogeneous Landau equation \eqref{Landau_Continuity} and the score-based Landau equation \eqref{Landau_Score}, respectively. Under the above assumptions, we have
    \begin{equation*}
        \frac{\rd}{\rd t} D_{\operatorname{KL}}(f_t ~||~ \tf_t)
        \leq C D_{\operatorname{KL}}(f_t ~||~ \tf_t) + \frac{4d^2}{\lambda_1} \|A\|_{L^\infty}^2 \int_{\bbT^d} f_t |\delta_t|^2 \rd\bv\,,
    \end{equation*}
    where $\delta_t(\bv):=\nabla\log f_t(\bv) - \bs_t(\bv)$ is the score-matching error, and
    $$C:=\frac{2d}{\lambda_1} \left\| \nabla \cdot A \right\|_{L^\infty}^2 + M\frac{2d^2}{\lambda_1} \left\| A \right\|_{L^\infty}^2$$
    is a constant independent of $f_t$ and $\tf_t$.
    Here, the divergence of a matrix function is applied row-wise.
\end{theorem}

Before proving this theorem, we require the following lemma to quantify the KL divergence between two probability densities, both satisfying the continuity equation but with different velocity fields.
\begin{lemma}\label{KL_equality}
    Let $f_t$ and $\tilde f_t$ be solutions to the following continuity equations, respectively: 
    \begin{equation*}
        \partial_t f_t + \nabla \cdot [\bU^\delta(f_t)f_t] = 0 \quad \text{and} \quad \partial_t \tf_t + \nabla \cdot [\bU(\tf_t)\tf_t] = 0 \,.
    \end{equation*}
    Then 
    \begin{equation*}
        \frac{\rd}{\rd t} D_{\operatorname{KL}}(f_t ~||~ \tf_t) = \int_{\bbT^d} f_t \nabla\log\left(\frac{f_t}{\tf_t}\right) [\bU^\delta(f_t)- \bU(\tf_t)] \rd\bv \,.
    \end{equation*}
\end{lemma}
\begin{proof}
The formula follows from direct calculation: 
\begin{flalign*}
    \frac{\rd}{\rd t} D_{\operatorname{KL}}(f_t ~||~ \tf_t) 
    & = \frac{\rd}{\rd t} \int_{\bbT^d} f_t \log\left(\frac{f_t}{\tf_t}\right) \rd\bv \\
    & = \int_{\bbT^d} \partial_t f_t \log\left(\frac{f_t}{\tf_t}\right) \rd\bv - \int_{\bbT^d} \left( \frac{f_t}{\tf_t} \right) \partial_t \tf_t \rd\bv \\
    & = -\int_{\bbT^d} \nabla \cdot [\bU^\delta(f_t)f_t] \log\left(\frac{f_t}{\tf_t}\right) \rd\bv + \int_{\bbT^d} \nabla \cdot [\bU(\tf_t)\tf_t] \left(\frac{f_t}{\tf_t}\right) \rd\bv \\
    & = \int_{\bbT^d} \bU^\delta(f_t)f_t \nabla\log\left(\frac{f_t}{\tf_t}\right) \rd\bv - \int_{\bbT^d} \bU(\tf_t)\tf_t \nabla\left(\frac{f_t}{\tf_t}\right) \rd\bv \\
    & = \int_{\bbT^d} \bU^\delta(f_t)f_t \nabla\log\left(\frac{f_t}{\tf_t}\right) \rd\bv - \int_{\bbT^d} \bU(\tf_t)f_t \nabla\log\left(\frac{f_t}{\tf_t}\right) \rd\bv \\
    & = \int_{\bbT^d} f_t \nabla\log\left(\frac{f_t}{\tf_t}\right) [\bU^\delta(f_t)- \bU(\tf_t)] \rd\bv \,.
\end{flalign*}
\end{proof}

We are now prepared to prove the main theorem. 
\begin{proof}[Proof of Theorem~\ref{KL_evolution}]
    Denote
    \[
    K(\bv):=\nabla \cdot A(\bv) = -2|\bv|^{\gamma} \bv \,,
    \]
    where the divergence of a matrix function is applied row-wise, the velocity field of \eqref{Landau_Continuity} can be rewritten as
    \begin{flalign*}
        \bU[\tf_t]
        & = -\int_{\bbT^d} A(\bv-\bv_*) \tf_{t,*} (\nabla\log\tf_t - \nabla_*\log\tf_{t,*}) \rd\bv_* \\
        & = -\int_{\bbT^d} A(\bv-\bv_*) \tf_{t,*} \nabla\log\tf_t \rd\bv_* + \int_{\bbT^d} A(\bv-\bv_*) \nabla_*\tf_{t,*} \rd\bv_* \\
        & = -\int_{\bbT^d} A(\bv-\bv_*) \tf_{t,*} \nabla\log\tf_t \rd\bv_* - \int_{\bbT^d} \nabla_* \cdot A(\bv-\bv_*) \tf_{t,*} \rd\bv_* \\
        & = -\int_{\bbT^d} A(\bv-\bv_*) \tf_{t,*} \nabla\log\tf_t \rd\bv_* + K*\tf_t \,.
    \end{flalign*}
    Here the convolution operation is applied entry-wise. Likewise, the velocity field of \eqref{Landau_Score} rewrites into
    \begin{flalign*}
        \bU^\delta  
        & = -\int_{\bbT^d} A(\bv-\bv_*) f_{t,*} (\bs_t - \bs_{t,*}) \rd\bv_* \\
        & = -\int_{\bbT^d} A(\bv-\bv_*) f_{t,*} (\nabla\log f_t - \nabla_*\log f_{t,*} - \delta_t + \delta_{t,*})\rd\bv_* \\
        & = -\int_{\bbT^d} A(\bv-\bv_*) f_{t,*} \nabla\log f_t \rd\bv_* + K*f_t + \int_{\bbT^d} A(\bv-\bv_*) f_{t,*}(\delta_t - \delta_{t,*})\rd\bv_* \,,
    \end{flalign*}
    where $\delta_t(\bv):=\nabla\log f_t(\bv) - \bs_t(\bv)$.
    By Lemma \ref{KL_equality}, we have
    \begin{flalign*}
        & \frac{\rd}{\rd t} D_{\operatorname{KL}}(f_t ~||~ \tf_t)
        = \int_{\bbT^d} f_t \nabla\log\left(\frac{f_t}{\tf_t}\right) \cdot \left(\bU^\delta  - \bU\right) \rd\bv \\
        = & - \int_{\bbT^d} f_t \nabla\log\left(\frac{f_t}{\tf_t}\right) \cdot \int_{\bbT^d} A(\bv-\bv_*) f_{t,*} \nabla\log f_t \rd\bv_*\rd\bv + \int_{\bbT^d} f_t \nabla\log\left(\frac{f_t}{\tf_t}\right) \cdot \int_{\bbT^d} A(\bv-\bv_*) \tf_{t,*} \nabla\log\tf_t \rd\bv_*\rd\bv \\
        & + \int_{\bbT^d} f_t \nabla\log\left(\frac{f_t}{\tf_t}\right) \cdot K*(f_t-\tf_t)\rd\bv + \int_{\bbT^d} f_t \nabla\log\left(\frac{f_t}{\tf_t}\right) \cdot \int_{\bbT^d} A(\bv-\bv_*) f_{t,*} (\delta_t - \delta_{t,*}) \rd\bv_*\rd\bv \\
        = & \underbrace{\int_{\bbT^d} f_t \nabla\log\left(\frac{f_t}{\tf_t}\right) \cdot \int_{\bbT^d} A(\bv-\bv_*) f_{t,*} \nabla\log \left(\frac{\tf_t}{f_t}\right) \rd\bv_*\rd\bv}_{I_1} 
        + \underbrace{\int_{\bbT^d} f_t \nabla\log\left(\frac{f_t}{\tf_t}\right) \cdot \int_{\bbT^d} A(\bv-\bv_*) (\tf_{t,*}-f_{t,*}) \nabla\log\tf_t \rd\bv_*\rd\bv}_{I_2} \\
        & + \underbrace{\int_{\bbT^d} f_t \nabla\log\left(\frac{f_t}{\tf_t}\right) \cdot K*(f_t-\tf_t) \rd\bv}_{I_3}
        + \underbrace{\int_{\bbT^d} f_t \nabla\log\left(\frac{f_t}{\tf_t}\right) \cdot \int_{\bbT^d} A(\bv-\bv_*) f_{t,*} (\delta_t - \delta_{t,*}) \rd\bv_*\rd\bv}_{I_4} \,.
    \end{flalign*}
    For $I_1$, using assumption (A.1), we have 
    \begin{flalign} \label{I1}
        I_1 
        & = \int_{\bbT^d} f_t \nabla\log\left(\frac{f_t}{\tf_t}\right) \cdot \int_{\bbT^d} A(\bv-\bv_*) f_{t,*} \nabla\log\left(\frac{\tf_t}{f_t}\right) \rd\bv_*\rd\bv \nonumber  \\
        & = -\int_{\bbT^d} \int_{\bbT^d} f_t f_{t,*} \nabla\log\left(\frac{f_t}{\tf_t}\right) A(\bv-\bv_*) \nabla\log \left(\frac{f_t}{\tf_t}\right) \rd\bv_*\rd\bv  \nonumber\\
        & \leq -\lambda_1 \int_{\bbT^d} f_t \left| \nabla\log\left(\frac{f_t}{\tf_t}\right) \right|^2 \rd\bv\,.
    \end{flalign}
    $I_3$ can be estimated as follows:
    \begin{flalign} \label{03}
        I_3 
        & = \int_{\bbT^d} f_t \nabla\log\left(\frac{f_t}{\tf_t}\right) K*(f_t-\tf_t) \rd\bv \nonumber  \\
        & \leq \frac{\lambda_1}{4} \int_{\bbT^d} f_t \left| \nabla\log\left(\frac{f_t}{\tf_t}\right) \right|^2 \rd\bv + \frac{1}{\lambda_1} \int_{\bbT^d} f_t \left| K*(f_t-\tf_t) \right|^2 \rd\bv \nonumber \\
        & \leq \frac{\lambda_1}{4} \int_{\bbT^d} f_t \left| \nabla\log\left(\frac{f_t}{\tf_t}\right) \right|^2 \rd\bv + \frac{1}{\lambda_1} \left\| | K*(f_t-\tf_t) |^2 \right\|_{L^\infty}\,.
    \end{flalign}
    By Young's convolution inequality on torus, one has that
    \begin{equation*}
        \left\| | K*(f_t-\tf_t) |^2 \right\|_{L^\infty} \leq d\left\| K \right\|_{L^\infty}^2 \left\| f_t-\tf_t \right\|_{L^1}^2.
    \end{equation*}
    Further by Csiszar-Kullback-Pinsker inequality on torus, one has that
    \begin{equation*}
        \left\| f_t-\tf_t \right\|_{L^1}^2 \leq 2D_{\operatorname{KL}}(f_t ~||~ \tf_t)\,.
    \end{equation*}
    Putting these into \eqref{03}, we have 
    \begin{equation*}
        I_3 \leq \frac{\lambda_1}{4} \int_{\bbT^d} f_t \left| \nabla\log\left(\frac{f_t}{\tf_t}\right) \right|^2 \rd\bv + \frac{2d}{\lambda_1} \left\| K \right\|_{L^\infty}^2 D_{\operatorname{KL}}(f_t ~||~ \tf_t)\,.
    \end{equation*}
    For $I_2$, we have that 
    \begin{flalign*}
        I_2
        & = \int_{\bbT^d} f_t \nabla\log\left(\frac{f_t}{\tf_t}\right) \int_{\bbT^d} A(\bv-\bv_*) (\tf_{t,*}-f_{t,*}) \nabla\log\tf_t \rd\bv_*\rd\bv \\
        & = \int_{\bbT^d} f_t \nabla\log\left(\frac{f_t}{\tf_t}\right) A*(\tf_t-f_t) \nabla\log\tf_t \rd\bv \\
        & \leq \frac{\lambda_1}{4} \int_{\bbT^d} f_t \left| \nabla\log\left(\frac{f_t}{\tf_t}\right) \right|^2 \rd\bv + \frac{1}{\lambda_1} \int_{\bbT^d} f_t \left| A*(f_t-\tf_t) \nabla\log\tf_t \right|^2 \rd\bv \\
        & \leq \frac{\lambda_1}{4} \int_{\bbT^d} f_t \left| \nabla\log\left(\frac{f_t}{\tf_t}\right) \right|^2 \rd\bv + M^2\frac{d^2}{\lambda_1} \left\| A \right\|_{L^\infty}^2 \left\| f_t - \tf_t \right\|_{L^1}^2 \\
        & \leq \frac{\lambda_1}{4} \int_{\bbT^d} f_t \left| \nabla\log\left(\frac{f_t}{\tf_t}\right) \right|^2 \rd\bv + M^2\frac{2d^2}{\lambda_1} \left\| A \right\|_{L^\infty}^2 D_{\operatorname{KL}}(f_t~||~\tf_t) \,,
    \end{flalign*}
    where the second inequality utilizes assumption (A.3), the third inequality applies Young's convolution inequality, and the final inequality employs the Csiszar-Kullback-Pinsker inequality.

    To estimate $I_4$, first we have 
    \begin{flalign*}
        I_4
        & = \int_{\bbT^d} f_t \nabla\log\left(\frac{f_t}{\tf_t}\right) \int_{\bbT^d} A(\bv-\bv_*) f_{t,*} (\delta_t - \delta_{t,*}) \rd\bv_*\rd\bv \\
        & = \int_{\bbT^d} f_t \nabla\log\left(\frac{f_t}{\tf_t}\right) [(A*f_t)\delta_t - A*(f_t\delta_t)] \rd\bv \\
        & \leq \frac{\lambda_1}{4} \int_{\bbT^d} f_t \left| \nabla\log\left(\frac{f_t}{\tf_t}\right) \right|^2 \rd\bv + \frac{2}{\lambda_1} \int_{\bbT^d} f_t \left| (A*f_t)\delta_t \right|^2\rd\bv + \frac{2}{\lambda_1} \int_{\bbT^d} f_t \left| A*(f_t\delta_t) \right|^2\rd\bv\,.
    \end{flalign*}
    Note that
    \begin{flalign*}
        \int_{\bbT^d} f_t \left| (A*f_t)\delta_t \right|^2 \rd\bv 
        \leq \sum_{i,j=1}^d \|A_{ij}*f_t\|_{L^\infty}^2 \int_{\bbT^d} f_t |\delta_t|^2 \rd\bv
        \leq \sum_{i,j=1}^d \|A_{ij}\|_{L^\infty}^2 \|f_t\|_{L^1}^2 \int_{\bbT^d} f_t |\delta_t|^2 \rd\bv 
        = d^2 \|A\|_{L^\infty}^2 \int_{\bbT^d} f_t |\delta_t|^2 \rd\bv
    \end{flalign*}
    and
    \begin{flalign*}
        & \int_{\bbT^d} f_t \left| A*(f_t \delta_t) \right|^2\rd\bv 
        \leq \| |A*(f_t \delta_t)|^2 \|_{L^\infty} \int_{\bbT^d} f_t \rd\bv
        \leq \sum_{i=1}^d \left\| \left|\sum_{j=1}^d A_{ij} * (f_t \delta_{t,j})\right|^2 \right\|_{L^\infty} \\
        & \leq d \sum_{i=1}^d \left\| \sum_{j=1}^d \left|A_{ij} * (f_t \delta_{t,j})\right|^2 \right\|_{L^\infty}
        \leq d \sum_{i,j=1}^d \left\|A_{ij} * (f_t \delta_{t,j})\right\|_{L^\infty}^2
        \leq d \sum_{i,j=1}^d \|A_{ij}\|_{L^\infty}^2 \|f_t \delta_{t,j}\|_{L^1}^2 \\
        & \leq d^2 \|A\|_{L^\infty}^2 \sum_{j=1}^d \|f_t \delta_{t,j}\|_{L^1}^2
        \leq d^2 \|A\|_{L^\infty}^2 \sum_{j=1}^d \int_{\bbT^d} f_t \delta_{t,j}^2 \rd\bv \int_{\bbT^d} f_t \rd\bv
        = d^2 \|A\|_{L^\infty}^2 \int_{\bbT^d} f_t |\delta_t|^2 \rd\bv\,.
    \end{flalign*}
    Hence
    \[
    I_4 \leq \frac{\lambda_1}{4} \int_{\bbT^d} f_t \left| \nabla\log\left(\frac{f_t}{\tf_t}\right) \right|^2 \rd\bv + \frac{4d^2}{\lambda_1} \|A\|_{L^\infty}^2 \int_{\bbT^d} f_t |\delta_t|^2 \rd\bv\,.
    \]

    Combining the inequalities for $I_1, I_2, I_3, I_4$, one obtains that
    \begin{flalign*}
        \frac{\rd}{\rd t} D_{\operatorname{KL}}(f_t ~||~ \tf_t) 
        & \leq -\frac{\lambda_1}{4} \int_{\bbT^d} f_t \left| \nabla\log\left(\frac{f_t}{\tf_t}\right) \right|^2 \rd\bv 
        + \left[ \frac{2d}{\lambda_1} \left\| K \right\|_{L^\infty}^2 + M^2 \frac{2d^2}{\lambda_1} \left\| A \right\|_{L^\infty}^2 \right] D_{\operatorname{KL}}(f_t ~||~ \tf_t)
        + \frac{4d^2}{\lambda_1} \|A\|_{L^\infty}^2 \int_{\bbT^d} f_t |\delta_t|^2 \rd\bv \\
        & \leq \underbrace{\left(\frac{2d}{\lambda_1} \left\| K \right\|_{L^\infty}^2 + M\frac{2d^2}{\lambda_1} \left\| A \right\|_{L^\infty}^2 \right)}_{C} D_{\operatorname{KL}}(f_t ~||~ \tf_t) + \frac{4d^2}{\lambda_1} \|A\|_{L^\infty}^2 \int_{\bbT^d} f_t |\delta_t|^2 \rd\bv \,.
    \end{flalign*}
    Here $C$ is a constant independent of $f_t$ and $\tf_t$.
\end{proof}

\begin{remark}
    One may improve the above time-dependent estimate to a uniform-in-time estimate by deriving a Logarithmic Sobolev inequality under additional assumptions, a task we defer to future investigations.
\end{remark}

\section{Exact density computation} \label{sec:4}
Although our algorithm (\ref{ISM_loss}--\ref{vup}) does not require density to advance in time, it is still advantageous to compute density since certain quantities, like entropy, rely on density values. In this section, we present an exact density computation formula by deriving the evolution equation for the logarithm of the determinant of the gradient of the transport map, i.e., $\det \nabla_\bV \bT(\bV,t)$. This equation, together with the change of variable formula,
\begin{equation*}
    \tf_t(\bT(\bV,t)) = \frac{f_0(\bV)}{|\det \nabla_\bV \bT(\bV,t)|}\,,
\end{equation*}
gives rise to the density along the particle trajectories. 

More precisely, recall the flow map $\bT(\cdot,t)$ corresponding to the Landau equation in \eqref{flow_map}: 
\begin{equation} \label{flow_map2}
    \frac{\rd}{\rd t} \bT(\bV,t) = -\underbrace{\int_{\R^d} A(\bT(\bV, t)-\bT(\bV_*, t)) \left[\nabla \log\tf_t (\bT(\bV, t)) - \nabla \log\tf_t(\bT(\bV_*, t))\right] f_0(\bV_*) \rd \bV_*}_{G(\bT(\bV, t))} \,.
\end{equation}
In practice, computing $\det \nabla_\bV T(\bV,t)$ poses a significant bottleneck due to its cubic cost with respect to the dimension of $\bV$. To address this issue, inspired by the concept of continuous normalizing flow \cite{neuralode, grathwohl2018scalable, lee2023deep}, we derive the following proposition concerning the computation of the evolution of the logarithm of the determinant. This extends the instantaneous change of variable formula beyond the classical continuity equation.

\begin{proposition}\label{logdet} Assume the transport map $\bT(\cdot, t)$ from \eqref{flow_map2} is invertible, then the log determinant of its gradient satisfies the following evolution equation:
\begin{equation} \label{00}
    \frac{\rd}{\rd t} \log |\det \nabla_\bV \bT(\bV,t)| = -\int_{\R^d} \nabla \cdot \left[A(\bT(\bV,t)-\bT(\bV_*,t)) \left(\tbs_t(\bT(\bV,t)) - \tbs_t(\bT(\bV_*,t))\right)\right] f_0(\bV_*) \rd\bV_* \,.
\end{equation}
For the particular form of $A$ in \eqref{A}, it becomes 
\begin{equation} \label{01}
    \begin{split}
        \frac{\rd}{\rd t} \log |\det \nabla_\bV \bT(\bV,t)|
        = & -C_{\gamma} \int_{\mathbb{R}^{d}} \big\{ |\bT(\bV,t)-\bT(\bV_*,t)|^{\gamma+2} \\ 
        & \nabla \cdot \left[\Pi(\bT(\bV,t)-\bT(\bV_*,t)) \left(\tbs_t(\bT(\bV,t)) - \tbs_t(\bT(\bV_*,t))\right)\right] \big\} f_0(\bV_*) \rd\bV_* \,. 
    \end{split}
\end{equation}
\end{proposition}
\begin{proof}
    Since we assume $\bT(\cdot,t)$ is invertible, i.e. $\det \nabla_\bV \bT(\bV,t) \not= 0$, then
    \begin{flalign*}
        \frac{\rd}{\rd t} \log |\det \nabla_\bV \bT(\bV,t)| 
        = & \operatorname{Tr}\left((\nabla_\bV \bT(\bV,t))^{-1} \frac{\rd}{\rd t} \nabla_\bV \bT(\bV,t)\right) \\
        = & \operatorname{Tr}\left((\nabla_\bV \bT(\bV,t))^{-1} \nabla_\bV \frac{\rd}{\rd t} \bT(\bV,t)\right) \\
        = & -\operatorname{Tr}\left((\nabla_\bV \bT(\bV,t))^{-1} \nabla_\bV G(\bT(\bV, t)) \right) \\
        = & -\operatorname{Tr}\left((\nabla_\bV \bT(\bV,t))^{-1} \nabla G(\bT(\bV, t)) \nabla_\bV \bT(\bV,t)\right) \\
        = & -\operatorname{Tr}\left(\nabla G(\bT(\bV, t))\right) \\
        = & -\nabla \cdot G(\bT(\bV, t)) \\
        = & -\int_{\R^d} \nabla \cdot \left[A(\bT(\bV,t)-\bT(\bV_*,t)) \left(\tbs_t(\bT(\bV,t)) - \tbs_t(\bT(\bV_*,t))\right)\right] f_0(\bV_*) \rd\bV_* \,.
    \end{flalign*}
    Here the first equality uses Jacobi's identity. To go from \eqref{00} to \eqref{01}, we use the projection property of the matrix $\Pi$ in the form of $A$.
\end{proof}

\begin{remark}
    If the collision kernel $A(z)=I_d$, then
    \begin{flalign*}
         \frac{\rd}{\rd t} \log |\det \nabla_\bV \bT(\bV,t)| 
         = -\int_{\R^d} \nabla \cdot \left(\tbs_t(\bT(\bV,t)) - \tbs_t(\bT(\bV_*,t))\right) f_0(\bV_*) \rd\bV_* 
         = -\nabla \cdot \tbs_t(\bT(\bV,t)) \,, 
    \end{flalign*}
    which reduces to the classical case, see for instance \cite[Equation (4c)]{lee2023deep}.
\end{remark}

For a more straightforward implementation, we compute the divergence term in Proposition \ref{logdet} analytically, and summarize it in the following corollary.
\begin{corollary}\label{logdet_coro}
Expanding the divergence term in equation \eqref{01}, we get:
\begin{equation*}
    \begin{split}
        \frac{\rd}{\rd t} \log |\det \nabla_\bV \bT(\bV,t)| 
        & = - \int_{\mathbb{R}^{d}} \big[ A(\bT(\bV,t)-\bT(\bV_*,t)) : \nabla \tbs_t(\bT(\bV,t))^\top -C_{\gamma} (d-1) \\
        & |\bT(\bV,t)-\bT(\bV_*,t)|^{\gamma} (\bT(\bV,t)-\bT(\bV_*,t)) \cdot \left(\tbs_t(\bT(\bV,t)) - \tbs_t(\bT(\bV_*,t))\right) \big] f_0(\bV_*) \rd\bV_* \,,
    \end{split}
\end{equation*}
where $A:B := \sum_{ij}A_{ij}B_{ij}$.
\end{corollary}
\begin{proof}
    The proof is given in Appendix \ref{app_proof}.
\end{proof}

As with the particle dynamics \eqref{particle0}, the log determinant obtained from the above formula admits a particle representation. Recall that we have $N$ particles started with velocities $\{\bV_i\}_{i=1}^N \stackrel{i.i.d.}\sim f_0$, and their later velocities are denoted as $\bv_i(t) :=\bT(\bV_i, t)$, then we have, along the trajectory of the $i$-th particle: 
\begin{equation}\label{logdet_particle}
    \begin{split}
     & \frac{\rd}{\rd t} \log |\det \nabla_\bV T(\bV_i,t)| 
     = - \frac{1}{N} \sum_{j=1}^N \big[ A(\bv_i(t)-\bv_j(t)) : \nabla\bs_t(\bv_i(t))^\top \\
     & \hspace{4cm} - C_{\gamma}(d-1)|\bv_i(t)-\bv_j(t)|^{\gamma} (\bv_i(t)-\bv_j(t)) \cdot (\bs_t(\bv_i(t)) - \bs_t(\bv_j(t))) \big] \,.
    \end{split}
\end{equation}
In practice, the time discretization of equation \eqref{logdet_particle} is also performed using the forward Euler method.

Now we summarize the procedure of the score-based particle method with density computation in Algorithm \ref{algorithm_densitiy}.

\begin{algorithm}[H]
\caption{Score-based particle method with density computation}
\label{algorithm_densitiy}
    \begin{algorithmic}[1]
         \REQUIRE $N$ initial particles $\{\bV_i\}_{i=1}^N \stackrel{i.i.d.}\sim f_0$ and densities evaluated at the location of particles $f_0(\bV_i)$; time step $\Delta t$ and the total number of time steps $N_T$; error tolerance $\delta$ for the initial score-matching; max iteration number $I_{\max}$ for the implicit score-matching.
        \ENSURE Score neural networks $\bs_{n-1}$, particles $\{\bv_i^n\}_{i=1}^N$, and density values $\{f_n(\bv_i^n)\}_{i=1}^N$ for all $n=1, \cdots, N_T$. 
        
        \FOR{$n=0,\cdots,N_T-1$}
        \STATE  Use Algorithm~\ref{algorithm_no_densitiy} (line 1--4, 9--14) to learn the score $\bs_n$.
            \FOR{$i=1,\cdots,N$}
                \STATE obtain $\bv^{n+1}_i$ from $\bv^n_i$ via \eqref{vup}.
                \STATE $l_i^{n+1} = -\Delta t \frac1N \sum_{j=1}^{N} \left[ A(\bv_i^n-\bv_j^n) : \nabla\bs_n(\bv_i^n)^\top - C_{\gamma}{(d-1)}|\bv_i^n-\bv_j^n|^{\gamma} (\bv_i^n-\bv_j^n) \cdot (\bs_n(\bv_i^n)-\bs_n(\bv_j^n)) \right]$.
                \STATE Compute density along particle trajectory: $f_{n+1}(\bv_i^{n+1}) = \frac{f_n(\bv_i^n)}{\exp(l_i^{n+1})}$.
            \ENDFOR
        \ENDFOR
    \end{algorithmic}
\end{algorithm}

\section{Numerical results}\label{sec:5}
In this section, we showcase several numerical examples using the proposed score-based particle method in both Maxwell and Coulomb cases. To visualize particle density and compare it with the analytical or reference solution, we employ two approaches. One involves kernel density estimation (KDE), akin to the blob method for the deterministic particle method, as outlined in \cite{carrillo2020particle}. This is achieved by convolving the particle solution with the Gaussian kernel $\psi_\varepsilon$: 
\begin{equation*}
    f_n^{\operatorname{kde}}(\bv) := \frac{1}{N}\sum_{i=1}^N \psi_\varepsilon(\bv-\bv_i^n) \,.
\end{equation*}
The other approach is to apply our Algorithm \ref{algorithm_densitiy}, enabling us to directly obtain the density from those evolved particles.

We would like to highlight the efficiency gain in our approach in contrast to \cite{carrillo2020particle}, which primarily lies in the computation of the score. In \cite{carrillo2020particle}, the score is adjusted as the gradient of the variational derivative of regularized entropy, involving a double sum post discretization: 
\begin{equation} \label{double}
    \nabla \frac{\delta \mathcal{H}^N_\varepsilon}{\delta \tf_t}(\bv_i(t)) := \sum_{l=1}^N h^d \nabla \psi_\varepsilon (\bv_i(t) -\bv_l^c) \log \left(\sum_{k=1}^N w_k \psi_\varepsilon(\bv_l^c-\bv_k(t)) \right) \,,
\end{equation}
where $h$ is the mesh size, $\bv_i^c$ is the center of each square of the mesh.
On the contrary, in our present approach, we obtain the score using a neural network trained via the score-matching technique, markedly amplifying efficiency. Details will be provided in Section \ref{sec5.4}.

\subsection{Example 1: 2D BKW solution for Maxwell molecules}\label{sec5.1}
For the initial two tests in this and next subsections, we consider the BKW solution in both 2D and 3D to verify the accuracy of our method. This solution stands as one of the few analytical solutions available for the Landau equation. Further details regarding this solution can be found in \cite[Appendix A]{carrillo2020particle}. 

Consider the collision kernel
\begin{equation*}
    A(\bz)=\frac{1}{16}(|\bz|^2I_d - \bz \otimes \bz) \,,
\end{equation*}
then an exact solution of \eqref{Landau_Homogeneous} can be given by 
\begin{equation*}
    \tf_t(\bv) = \frac{1}{2\pi K} \exp\left( -\frac{|\bv|^2}{2K} \right)\left( \frac{2K-1}{K} + \frac{1-K}{2K^2}|\bv|^2 \right) \,, \quad K=1-\exp(-t/8)/2 \,.
\end{equation*}

\noindent \textit{Setting}. \quad
In the experiment, we set $t_0=0$ and compute the solution until $t_{\operatorname{end}}=5$. The time step is $\Delta t=0.01$. The total number of particles are set to be $N=150^2$, initially i.i.d. sampled from $f_0(\bv)=\frac{1}{\pi}|\bv|^2 e^{-|\bv|^2}$. We use rejection sampling to generate samples in the first quadrant and fill the other quadrants by symmetry. The score $\bs_t$ is parameterized as a fully-connected neural network with $3$ hidden layers, $32$ neurons per hidden layer, and \texttt{swish} activation function \cite{ramachandran2017searching}. The biases of the hidden layer are set to zero initially, whereas the weights of the hidden layer are initialized using a truncated normal distribution with a variance of 1/\texttt{fan\_in}, in accordance with the recommendations outlined in \cite{ioffe15}. We train the neural networks using \texttt{Adamax} optimizer \cite{kingma2017adam} with a learning rate of $\eta=10^{-4}$, loss tolerance $\delta=5 \times 10^{-5}$ for the initial score-matching, and the max iteration number $I_{\max}=25$ for the following implicit score-matching.  
\\
\\
\noindent \textit{Comparison}. \quad 
We first compare the learned score with the analytical score
\begin{equation*}
    \nabla \log \tf_t(\bv) = \left( -\frac{1}{K} + \frac{1-K}{(2K-1)K + \frac{1-K}{2}|\bv|^2} \right) \bv \,.
\end{equation*}
In Fig. \ref{fig1_score}, we present scatter plots of the learned score and analytical score from different viewing angles at time $t=1$ and $t=5$. Here, $score_x$ and $score_y$ refer to the $x$ and $y$ components of the score function, respectively. The locations of particles at different time are displayed in Fig. \ref{fig_1_par_loc}.

To measure the accuracy of the learned score over time, we measure the goodness of fit using the relative Fisher divergence defined by
\begin{equation} \label{rF}
    \frac{\int_{\R^d} | \bs_t(\bv) - \nabla \log \tf_t(\bv) |^2 f_t(\bv) \rd\bv}{\int_{\R^d} | \nabla\log \tf_t(\bv) |^2 f_t(\bv) \rd\bv} 
    \approx
    \frac{\sum_{i=1}^N | \bs_n(\bv_i^n) - \nabla \log \tf_{t_n}(\bv_i^n) |^2}{\sum_{i=1}^N | \nabla\log \tf_{t_n}(\bv_i^n) |^2} \,.
\end{equation}
This metric is plotted in Fig. \ref{fig1_fisher}, which demonstrates that the learned score closely matches the analytical score throughout the simulation duration. 

To further check the conservation and entropy decay properties of the method, we plot the time evolution of the kinetic energy in Fig. \ref{fig1_energy} and the entropy decay rate in Fig. \ref{fig1_entropy}. The energy is conserved up to a small error while the entropy decay rate matches the analytical one (computed using \eqref{entropy_decay_rate} with quadrature rule).

Furthermore, as is done in \cite{carrillo2020particle}, we reconstruct the solution via KDE on the computational domain $[-L, L]^2$ with $L = 4$, and uniformly divide the domain into $n=100^2$ meshes $\{Q_i\}_{i=1}^n$. The bandwidth of the Gaussian kernel is chosen to be $\varepsilon=0.15$. In Fig. \ref{fig1_L2}, we track the discrete relative $L^2$-error between the analytical solution and the reconstructed solution defined by
\begin{equation} \label{L2}
    \frac{\sqrt{\sum_{i=1}^n |f_t^{\operatorname{kde}}(\bv_i^c) - \tf_t(\bv_i^c)|^2}}{\sqrt{\sum_{i=1}^n |\tf_t(\bv_i^c)|^2}} \,,
\end{equation}
where $\bv_i^c$ is the center of $Q_i$. Finally, in Fig. \ref{fig1_kde}, we depict the slices of the solutions at different time. The plots demonstrate a close alignment between the reconstructed solutions and the analytical solutions.

\begin{figure}
    \centerline{\includegraphics[scale=0.55]{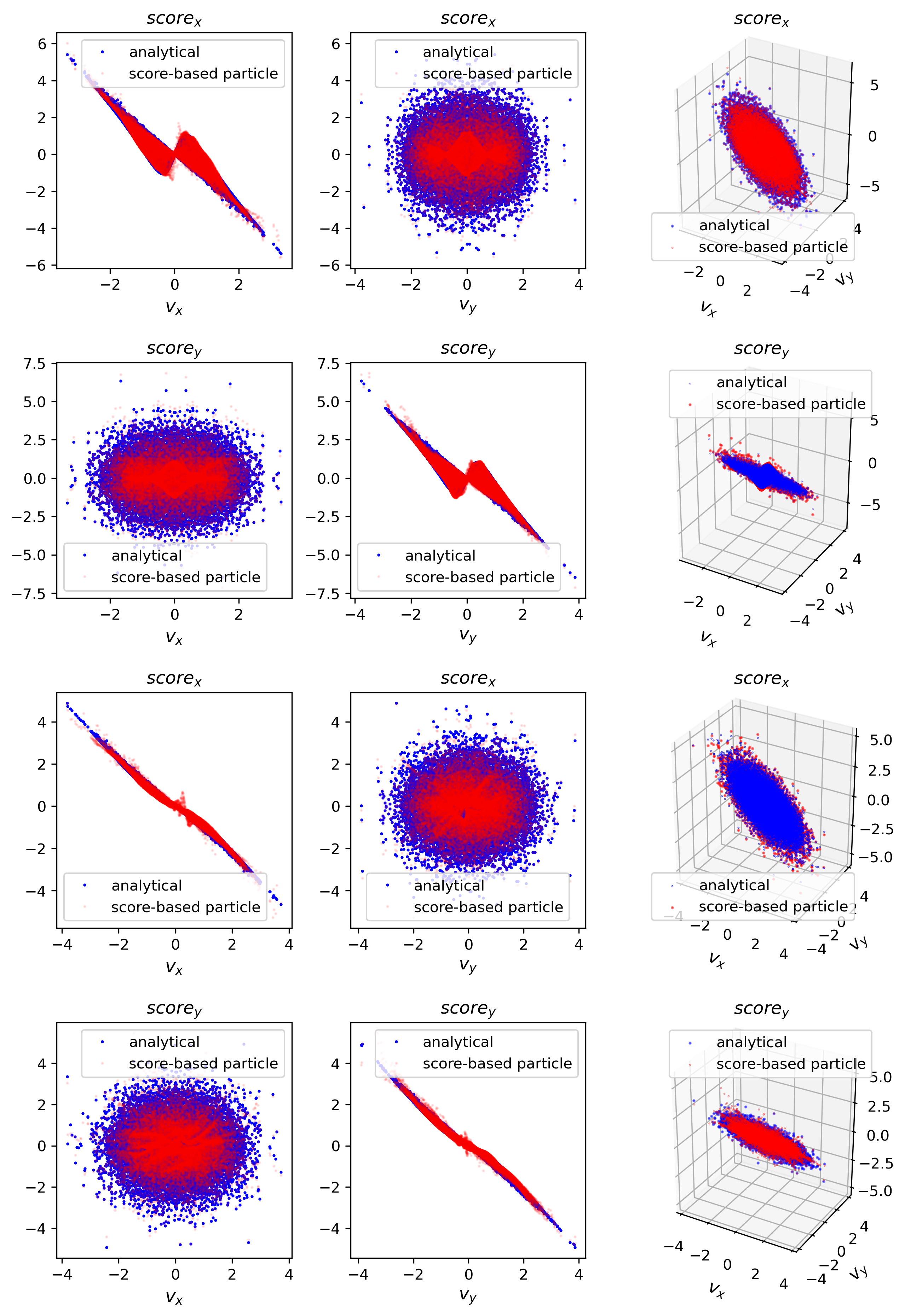}}
    \caption{Score visualization. The first two rows are at $t=1$, the last two rows are at $t=5$. }
    \label{fig1_score}
\end{figure}

\begin{figure}
    \centerline{\includegraphics[scale=0.5]{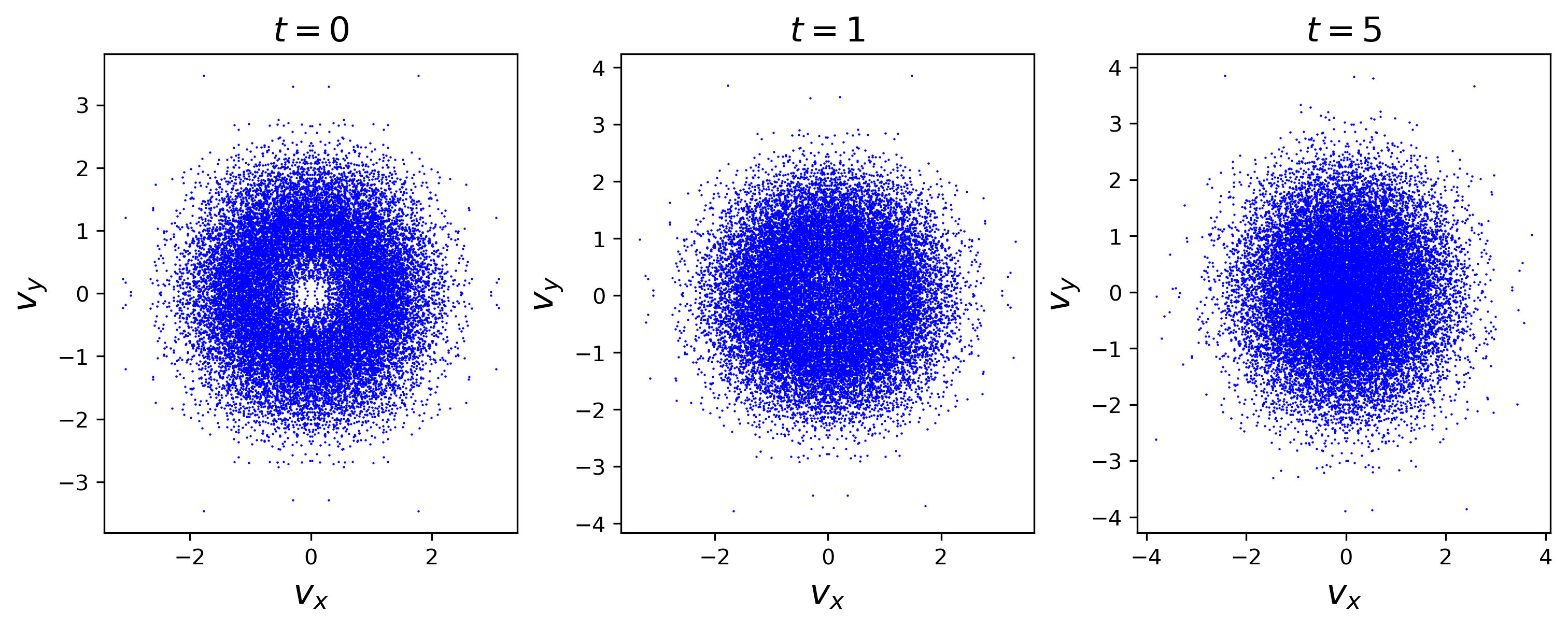}}
    \caption{The location of particles at $t=0, 1$, and $5$}.
    \label{fig_1_par_loc}
\end{figure}

\begin{figure}
    \centering
    \subfloat[Time evolution of the relative Fisher divergence \eqref{rF}.]{{\includegraphics[scale=0.42]{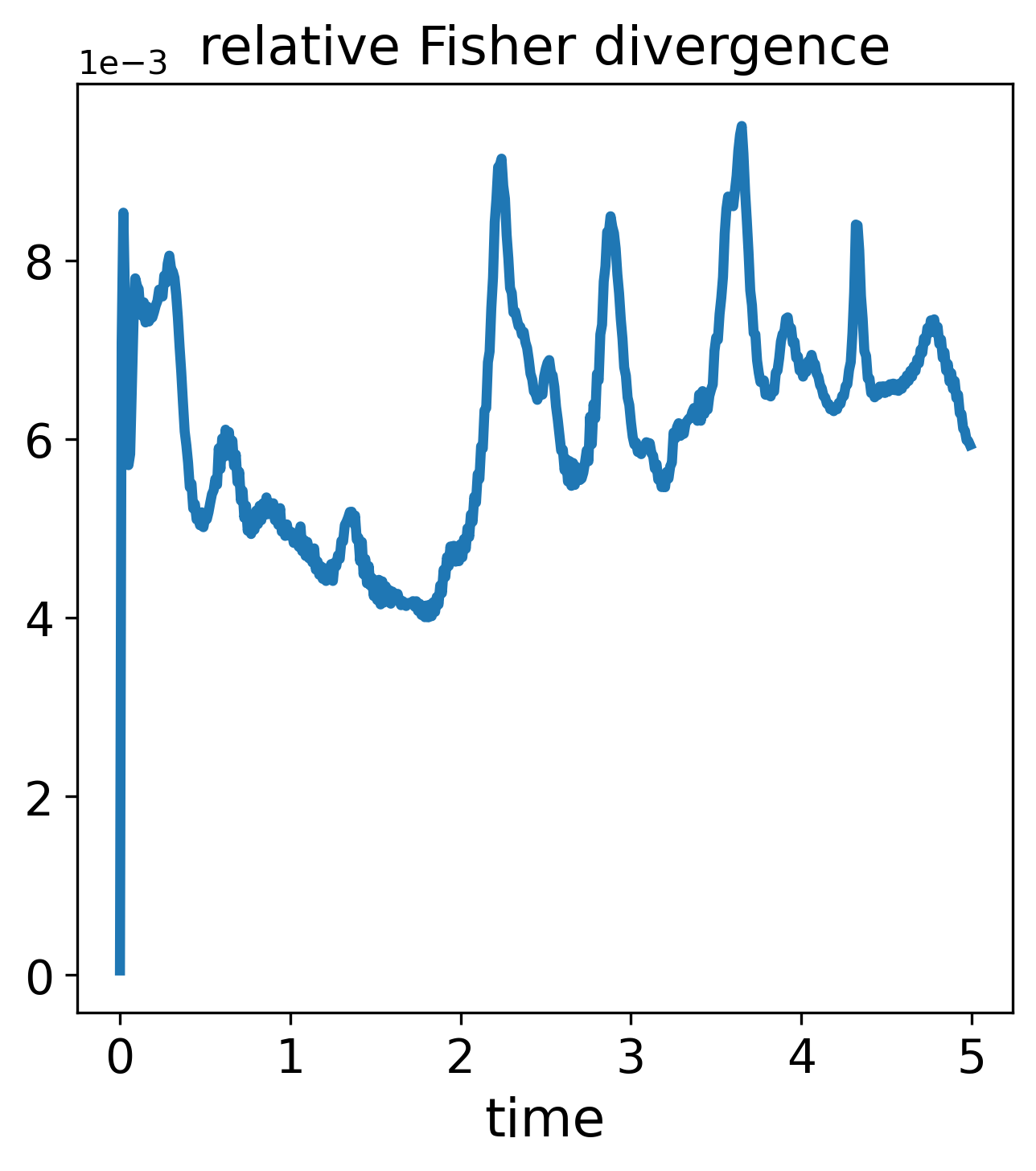}}\label{fig1_fisher}}
    \hspace{4em}
    \subfloat[Time evolution of the kinetic energy. ]{{\includegraphics[scale=0.42]{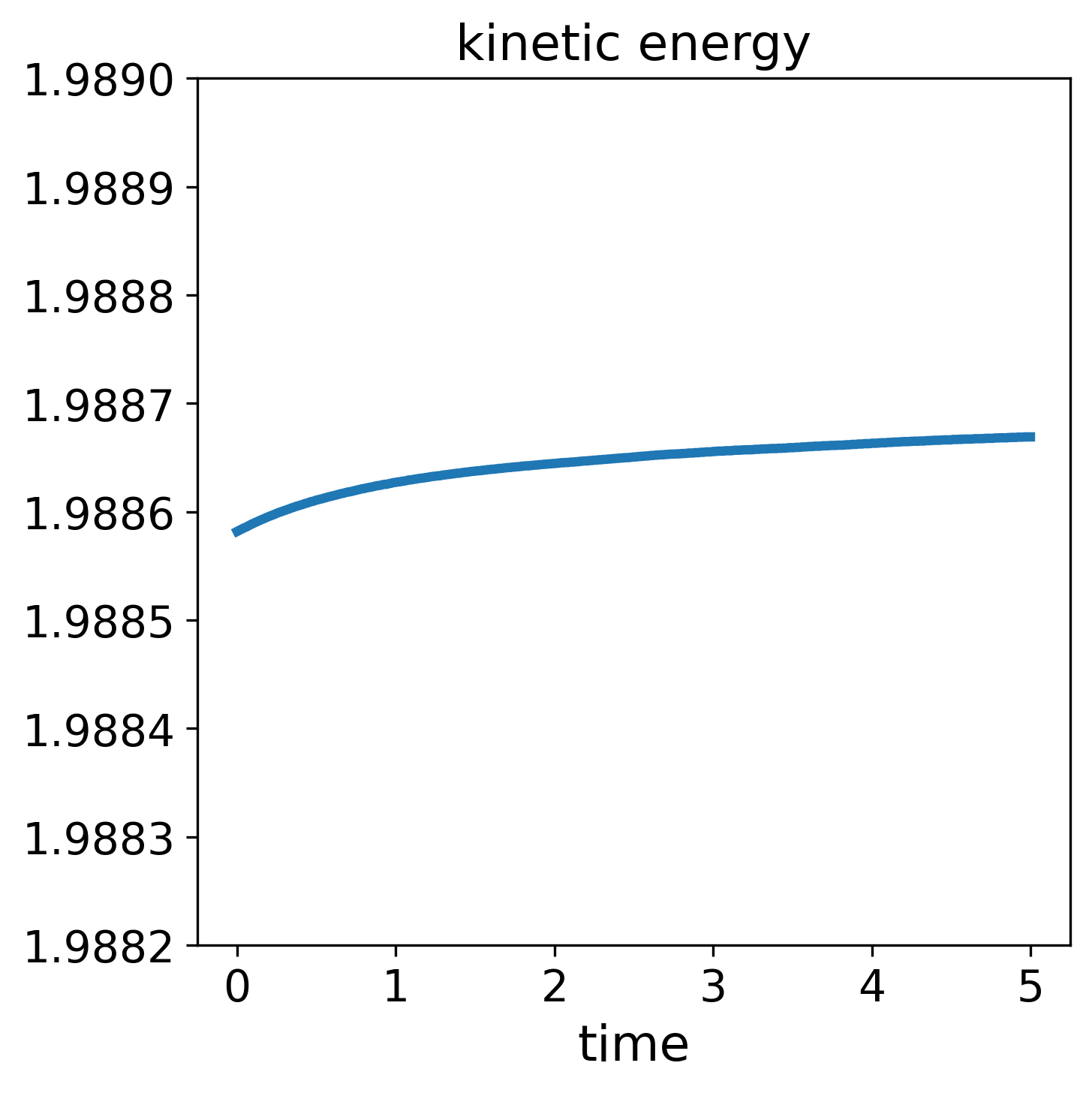}}\label{fig1_energy}}
    
    \vspace{1em}
    
    \subfloat[The evolution of the entropy decay rate.]{{\includegraphics[scale=0.42]{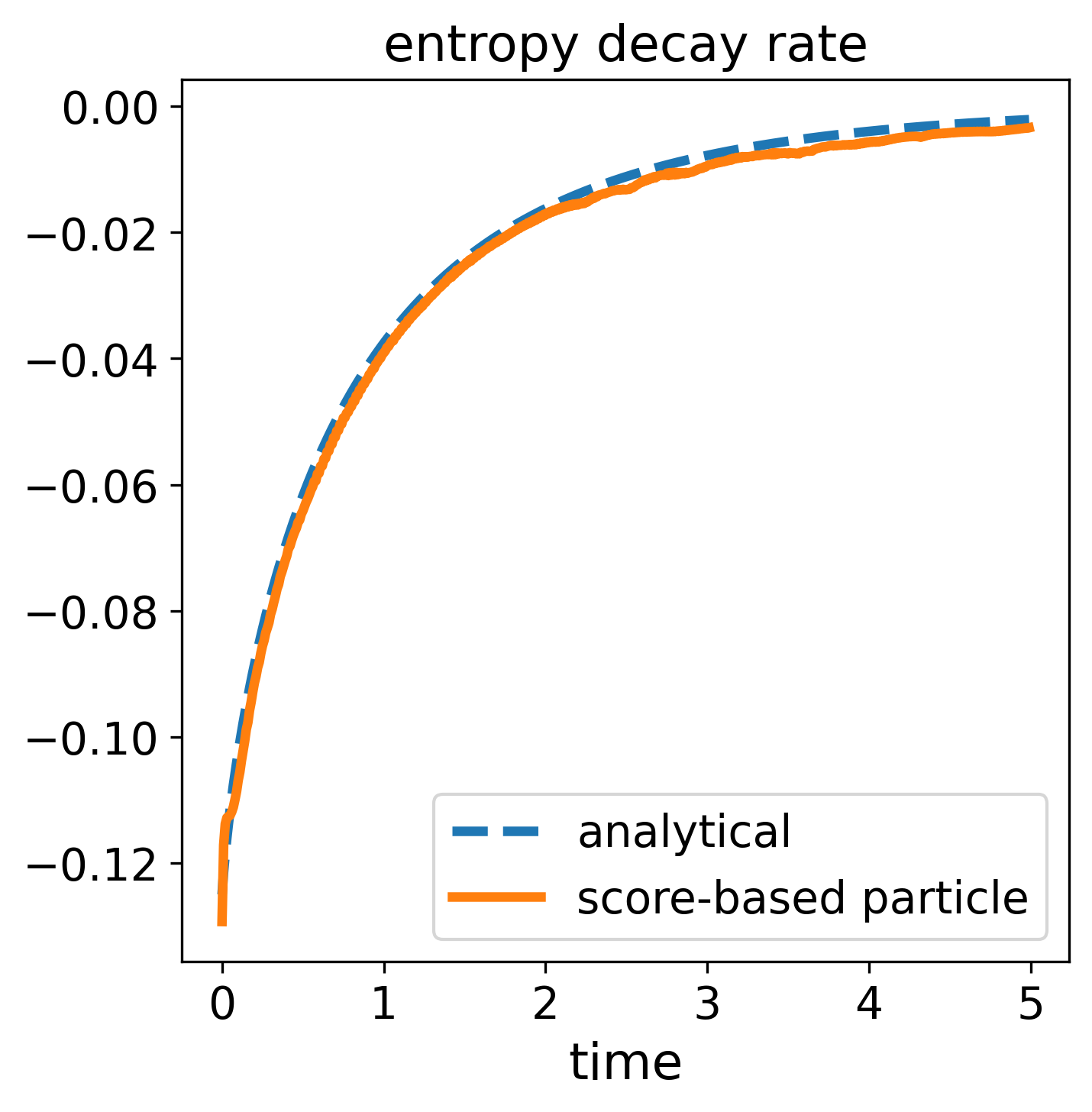}}\label{fig1_entropy}}
    \hspace{4em}
    \subfloat[Time evolution of the relative $L^2$-error \eqref{L2} of reconstructed density.]{{\includegraphics[scale=0.42]{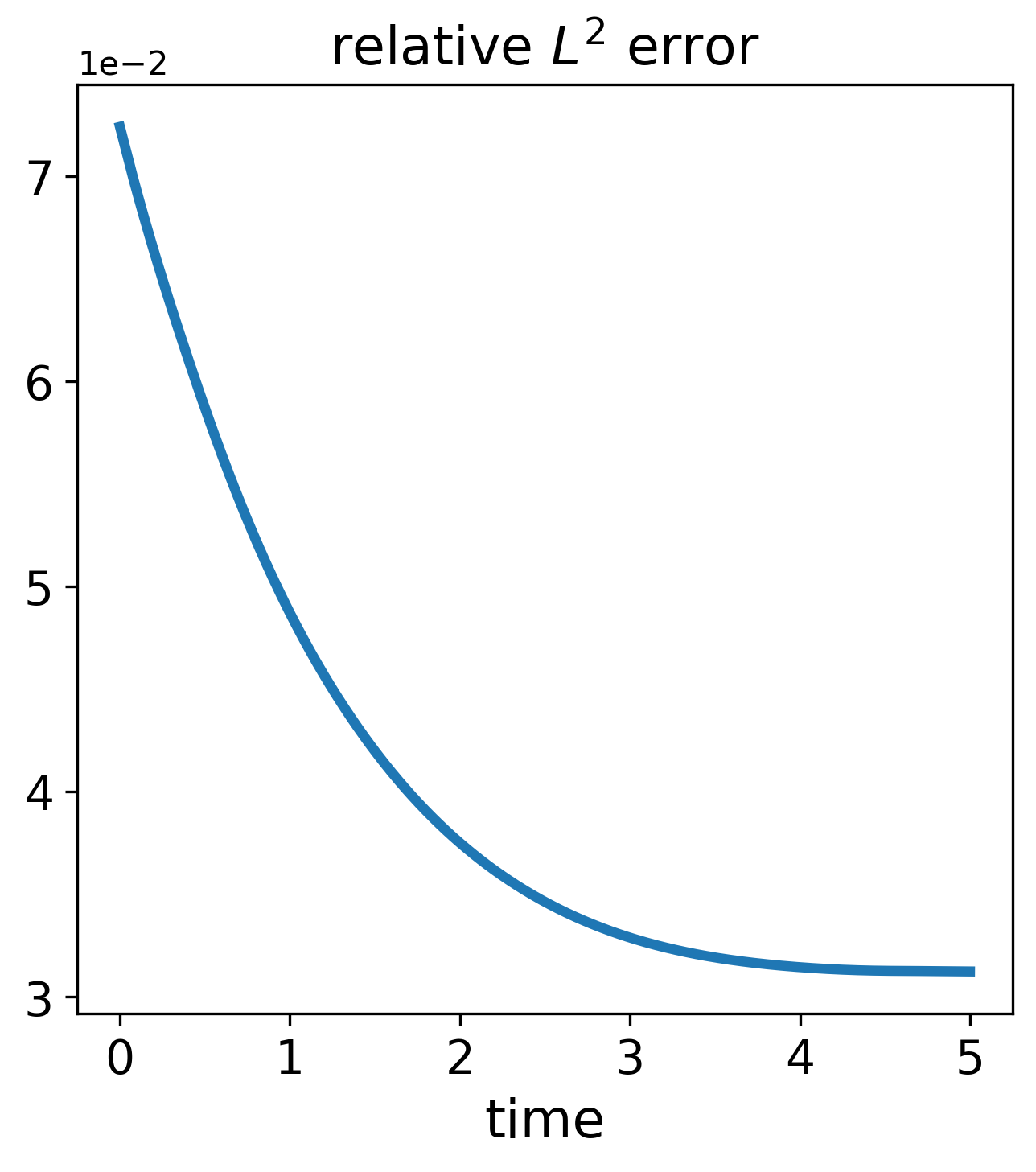}}\label{fig1_L2}}
    \caption{Quantitative comparisons between the numerical solution and analytical solution.}
\end{figure}

\begin{figure}
    \centerline{\includegraphics[scale=0.55]{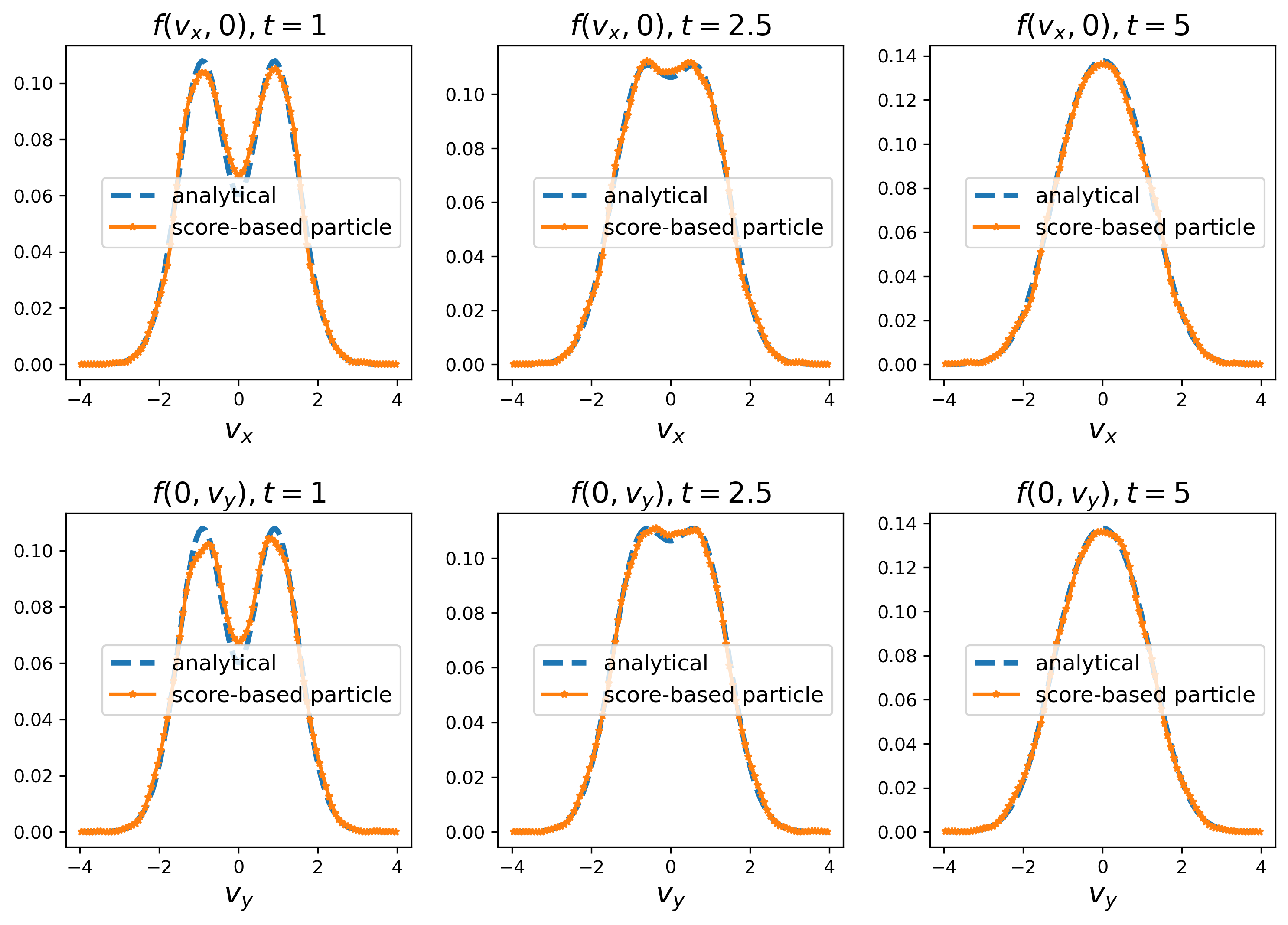}}
    \caption{Slices of the reconstructed and analytical solutions at $t=1$, $2.5$, and $5$.}
    \label{fig1_kde}
\end{figure}

\subsection{Example 2: 3D BKW solution for Maxwell molecules}
Consider the collision kernel
\begin{equation*}
    A(\bz)=\frac{1}{24}(|\bz|^2I_d - \bz \otimes \bz) \,,
\end{equation*}
and an exact solution is given by
\begin{equation*}
    \tf_t(\bv) = \frac{1}{(2\pi K)^{3/2}} \exp\left( -\frac{|\bv|^2}{2K} \right)\left( \frac{5K-3}{2K} + \frac{1-K}{2K^2}|\bv|^2 \right) \,, \quad K=1-\exp(-t/6) \,.
\end{equation*}
\textit{Setting}. \quad
In this test, we set $t_0=5.5$ and compute the solution until $t_{\operatorname{end}}=6$. The time step is $\Delta t=0.01$. The total number of particles is $N=40^3$, initially i.i.d. sampled from $\tf_{5.5}$ by rejection sampling. The architecture of score $\bs_t$ is a fully-connected neural network with $3$ hidden layers, $32$ neurons per hidden layer, and \texttt{swish} activation function. The initialization is identical to the first example. We train the neural networks using \texttt{Adamax} optimizer with a learning rate of $\eta=10^{-4}$, loss tolerance $\delta=10^{-4}$ for the initial score-matching, and the max iteration number $I_{\max}=25$ for the following implicit score-matching. 
\\
\\
\textit{Comparison}. \quad
The time evolution of the relative Fisher divergence, the kinetic energy, and the entropy decay rate are shown in Fig. \ref{fig_2_fisher}, \ref{fig_2_energy}, and \ref{fig_2_entropy}.

\begin{figure}[H]
    \centering
    \subfloat[Time evolution of the relative Fisher divergence \eqref{rF}.]{{\includegraphics[scale=0.45]{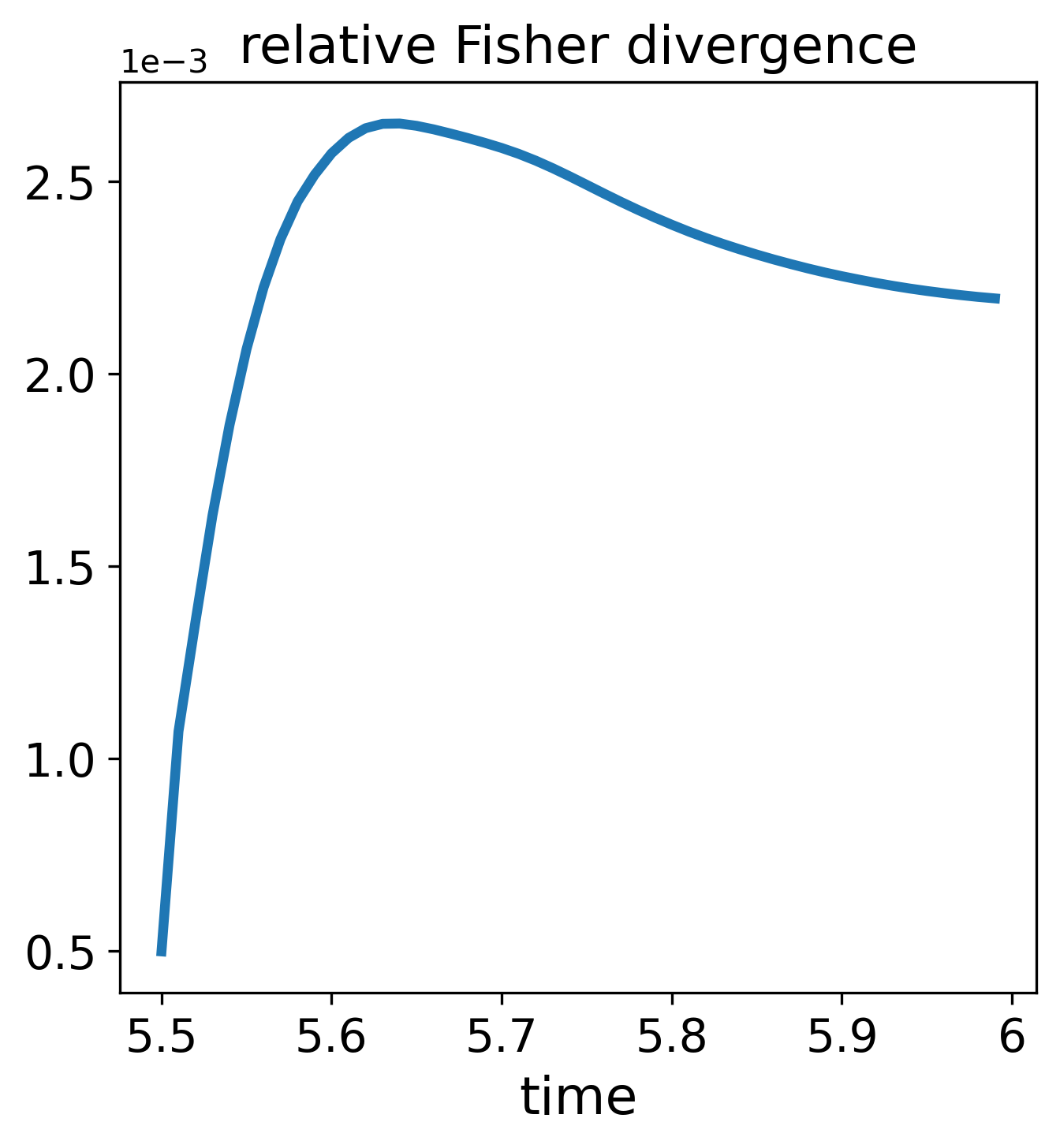}}\label{fig_2_fisher}}
    \hspace{4em}
    \subfloat[Time evolution of the kinetic energy.]{{\includegraphics[scale=0.45]{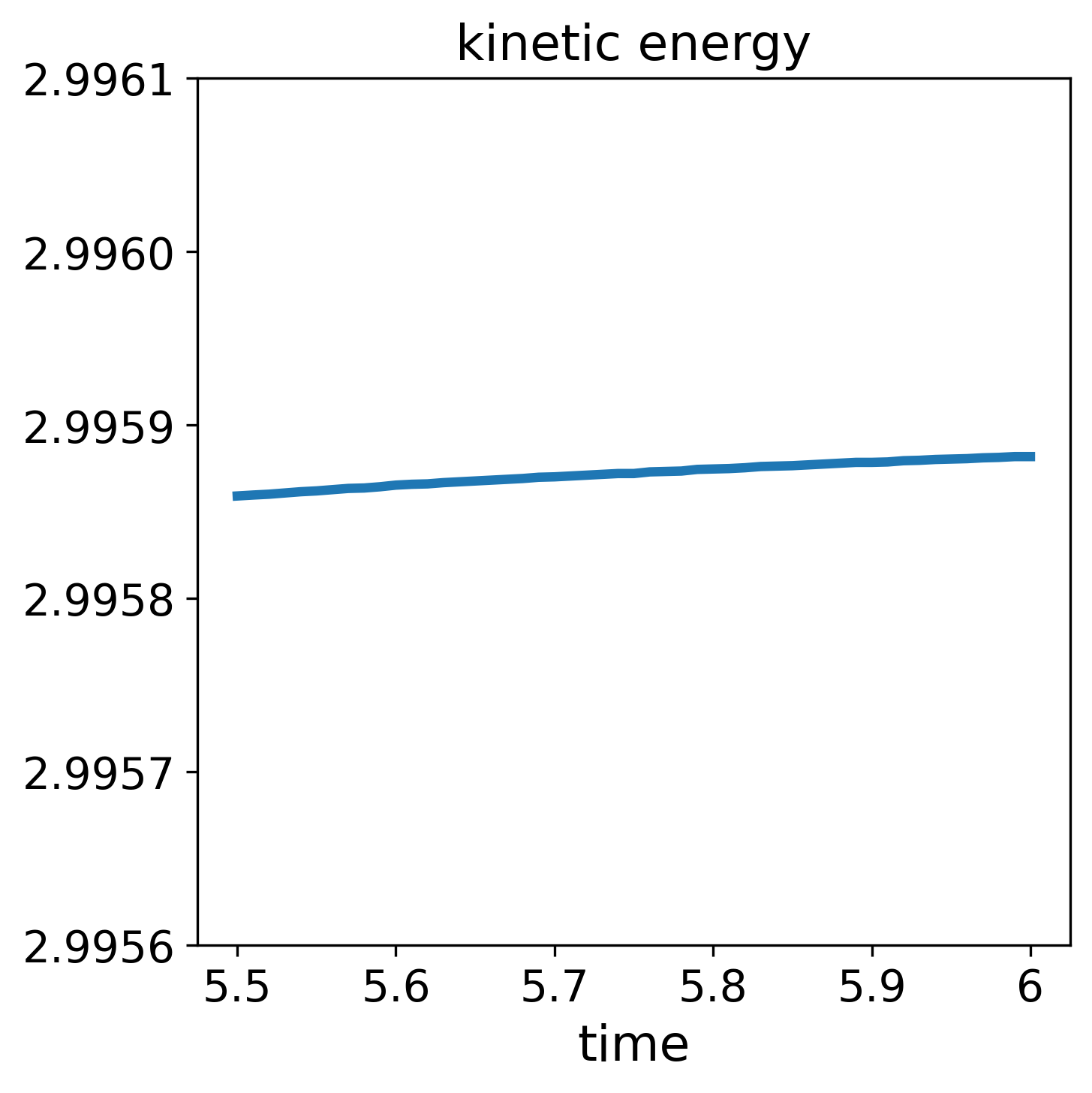}}\label{fig_2_energy}}
    
    \vspace{1em}
    
    \subfloat[The evolution of the entropy decay rate.]{{\includegraphics[scale=0.45]{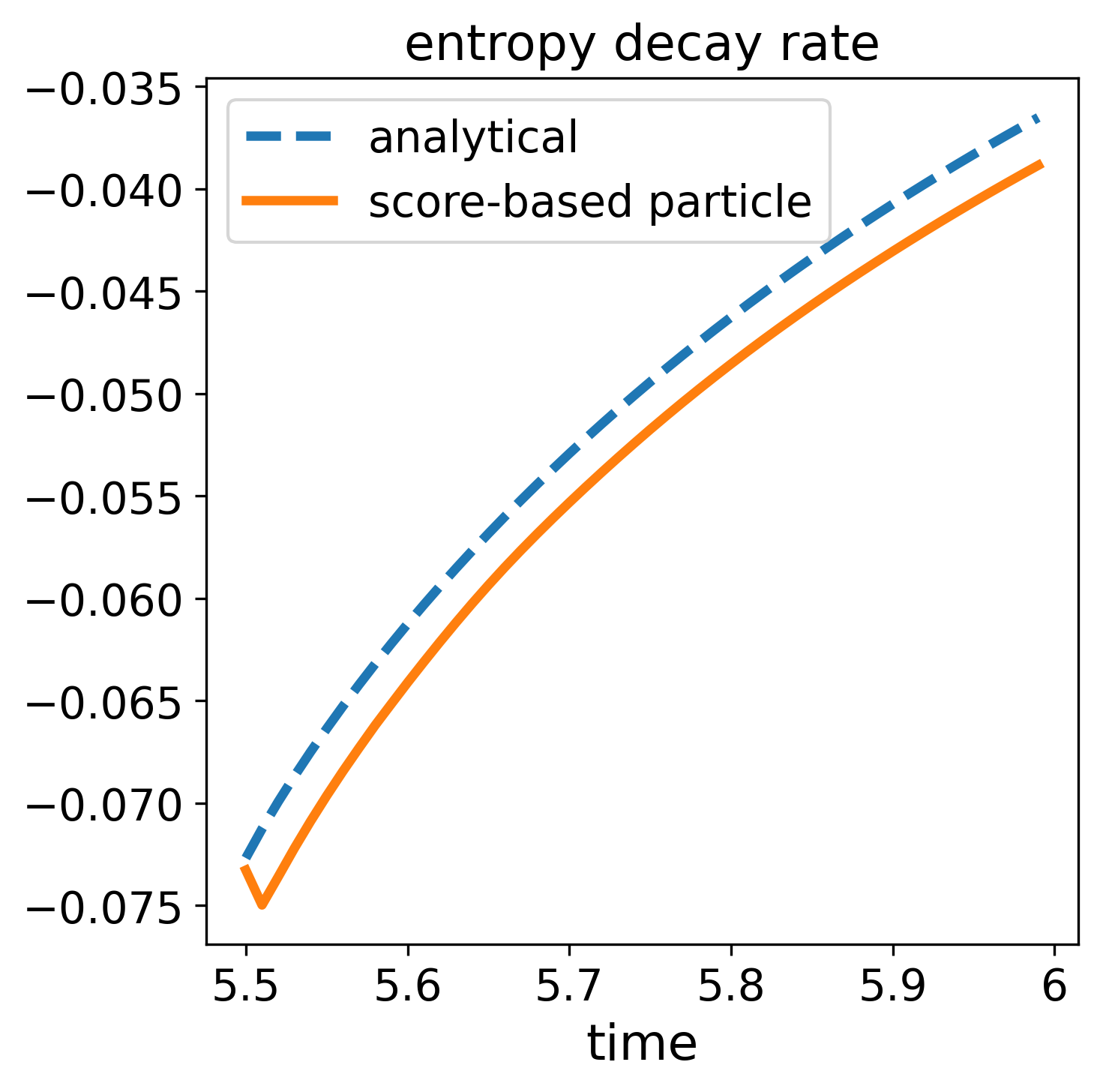}}\label{fig_2_entropy}}
    \hspace{4em}
    \subfloat[Time evolution of the relative $L^2$-error \eqref{L2} of reconstructed density.]{{\includegraphics[scale=0.45]{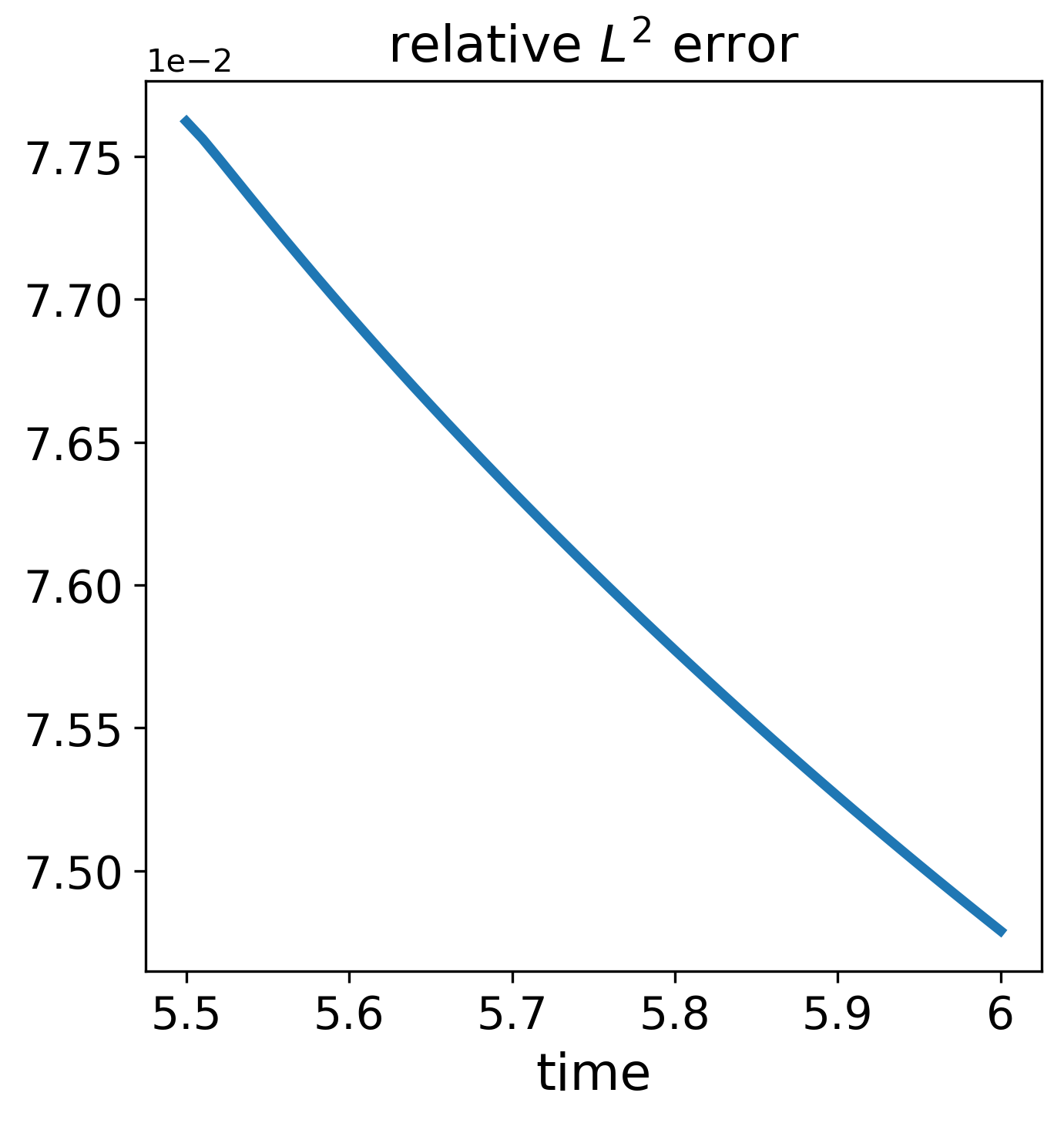}}\label{fig_2_L2}}
    \caption{Quantitative comparisons between the numerical solution and analytical solution.}
\end{figure}

We also reconstruct the solution via KDE on the computational domain $[-L, L]^3$ with $L = 4$, and uniformly divide the domain into $40^3$ meshes. The bandwidth of the Gaussian kernel is chosen to be $\varepsilon=0.15$. In Fig. \ref{fig_2_L2}, we track the discrete relative $L^2$-error (defined in \eqref{L2}) between the analytical solution and the reconstructed solution and plot a slice of the solution at $t=5.5, 5.75$ and $6$ in Fig. \ref{fig_2_kde}.
\begin{figure}
    \centering
    \includegraphics[scale=0.6]{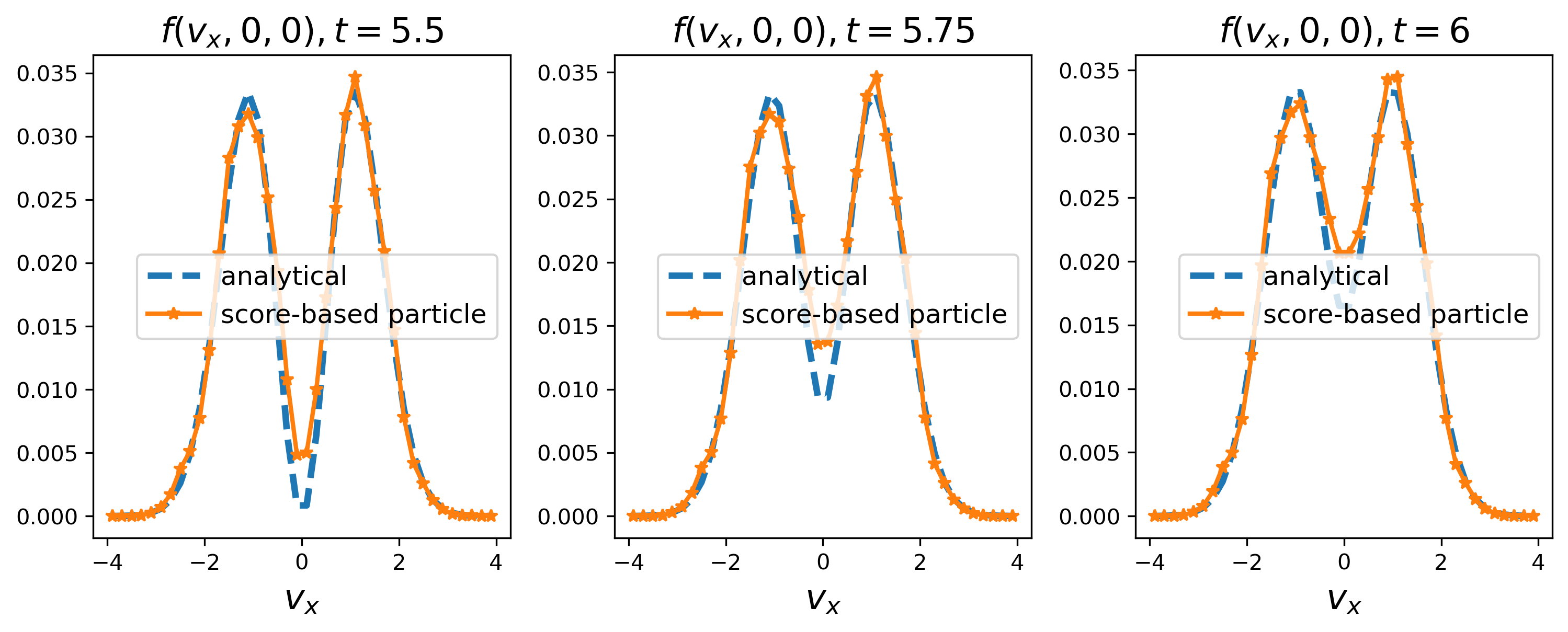}
    \caption{Slices of the reconstructed and analytical solution at $t=5.5$, $5.75$ and $6$.}
    \label{fig_2_kde}
\end{figure}

\subsection{Example 3: 2D anisotropic solution with Coulomb potential}
Consider the Coulomb collision kernel
\begin{equation*}
    A(\bz)=\frac{1}{16}\frac{1}{|\bz|^3}(|\bz|^2I_d - \bz \otimes \bz) \,,
\end{equation*}
and the initial condition is given by a bi-Maxwellian
\begin{equation*}
    f_0(\bv) = \frac{1}{4\pi} \left\{ \exp\left( -\frac{|\bv- \bu_1|^2}{2} \right) + \exp\left( -\frac{|\bv-\bu_2|^2}{2} \right)  \right\} \,, ~~ \bu_1=(-2,1) \,, ~~ \bu_2=(0,-1) \,.
\end{equation*}

\noindent \textit{Setting}. \quad
In this experiment, we start from $t_0=0$ and compute the solution until $t_{\operatorname{end}}=40$, with time step $\Delta t=0.1$. We choose the number of particles as $N=120^2$, sampled from $f_0$. 
We use a fully-connected neural network with $2$ hidden layers, $32$ neurons per hidden layer, and \texttt{swish} activation function to approximate the score $\bs_t$. The initialization is identical to the first example. We train the neural networks using \texttt{Adamax} optimizer with a learning rate of $\eta=10^{-4}$, loss tolerance $\delta=10^{-5}$ for the initial score-matching, and the max iteration number $I_{\max}=25$ for the following implicit score-matching.
\\
\\
\noindent \textit{Result}. \quad
We reconstruct the solution via KDE on the computational domain $[-L, L]^2$ with $L = 10$, and uniformly divide the domain into {$120^2$} meshes. The bandwidth of the Gaussian kernel is chosen to be $\varepsilon=0.3$. The result in Fig. \ref{fig3_kde} closely resembles the findings in the deterministic particle method and even the spectral method presented in \cite{carrillo2020particle}. 
\begin{figure}[H]
    \centerline{\includegraphics[scale=0.5]{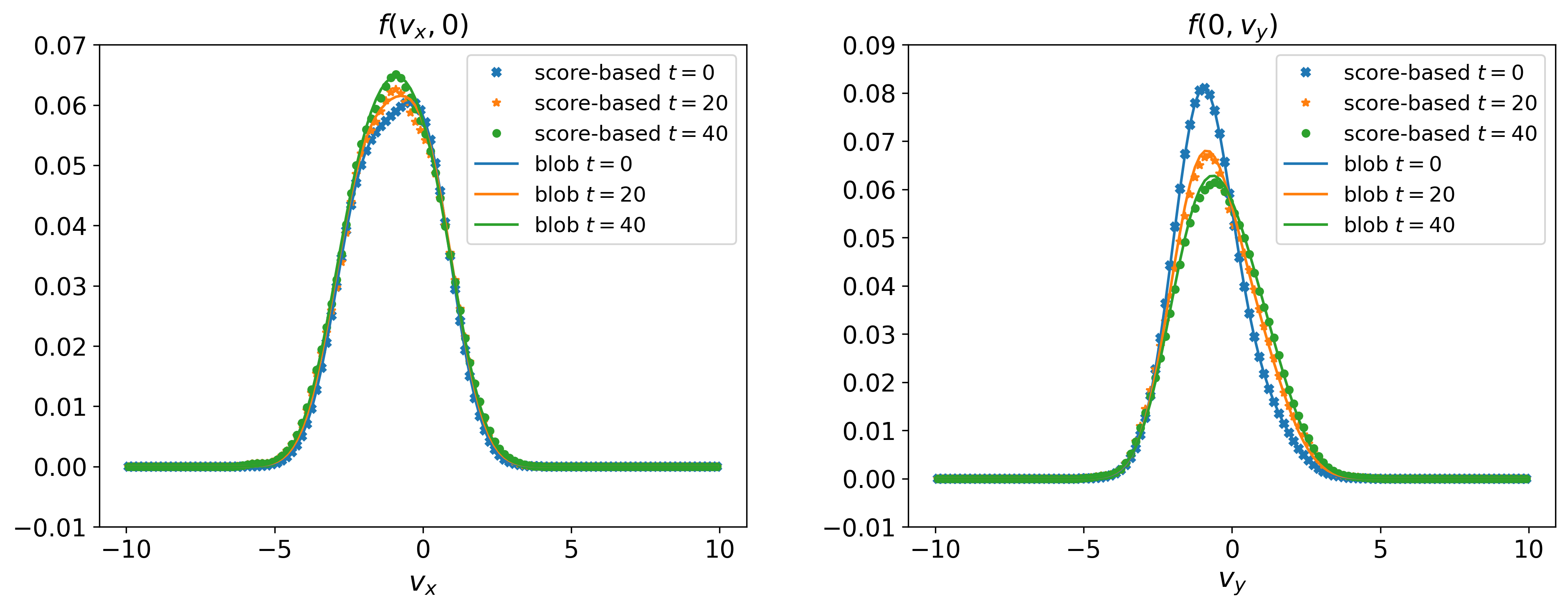}}
    \caption{Slices of the reconstructed and blob solutions at $t=0$, $20$, and $40$.}
    \label{fig3_kde}
\end{figure}

\subsection{Example 4: 3D Rosenbluth problem with Coulomb potential}\label{sec5.4}
Consider the collision kernel
\begin{equation*}
    A(\bz)=\frac{1}{4\pi}\frac{1}{|\bz|^3}(|\bz|^2I_d - \bz \otimes \bz) \,,
\end{equation*}
and the initial condition
\begin{equation*}
    f_0(\bv) = \frac{1}{S^2} \exp\left( -S\frac{(|\bv|-\sigma)^2}{\sigma^2} \right) \,, \quad \sigma=0.3 \,, \quad S=10 \,.
\end{equation*}

\noindent \textit{Setting}. \quad
In the example, we start from $t_0=0$ and compute the solution until $t_{\operatorname{end}}=20$, with time step $\Delta t=0.2$.  $N=30^3$ are initially sampled from $f_0$ by rejection sampling. The neural network approximating the score $\bs_t$ is set to be a residue neural network \cite{resnet} with $3$ hidden layers, $32$ neurons per hidden layer, and \texttt{swish} activation function. The initialization is identical to the first example. We train the neural networks using {\texttt{Adam}} optimizer with a learning rate of $\eta=10^{-4}$, loss tolerance $\delta=5\times 10^{-4}$ for the initial score-matching, and the max iteration number $I_{\max}=25$ for the following implicit score-matching.
\\
\\
\textit{Result}. \quad
We reconstruct the solution via KDE on the computational domain $[-L, L]^3$ with $L = 1$, and uniformly divide the domain into $64^3$ meshes. The bandwidth of the Gaussian kernel is chosen to be $\varepsilon=0.035$ for $t=10$ and $\varepsilon=0.045$ for $t=20$. In Fig. \ref{fig_eg4_kde}, we again observe a favorable agreement with the results presented in \cite{carrillo2020particle}.
\begin{figure}
    \centerline{\includegraphics[scale=0.55]{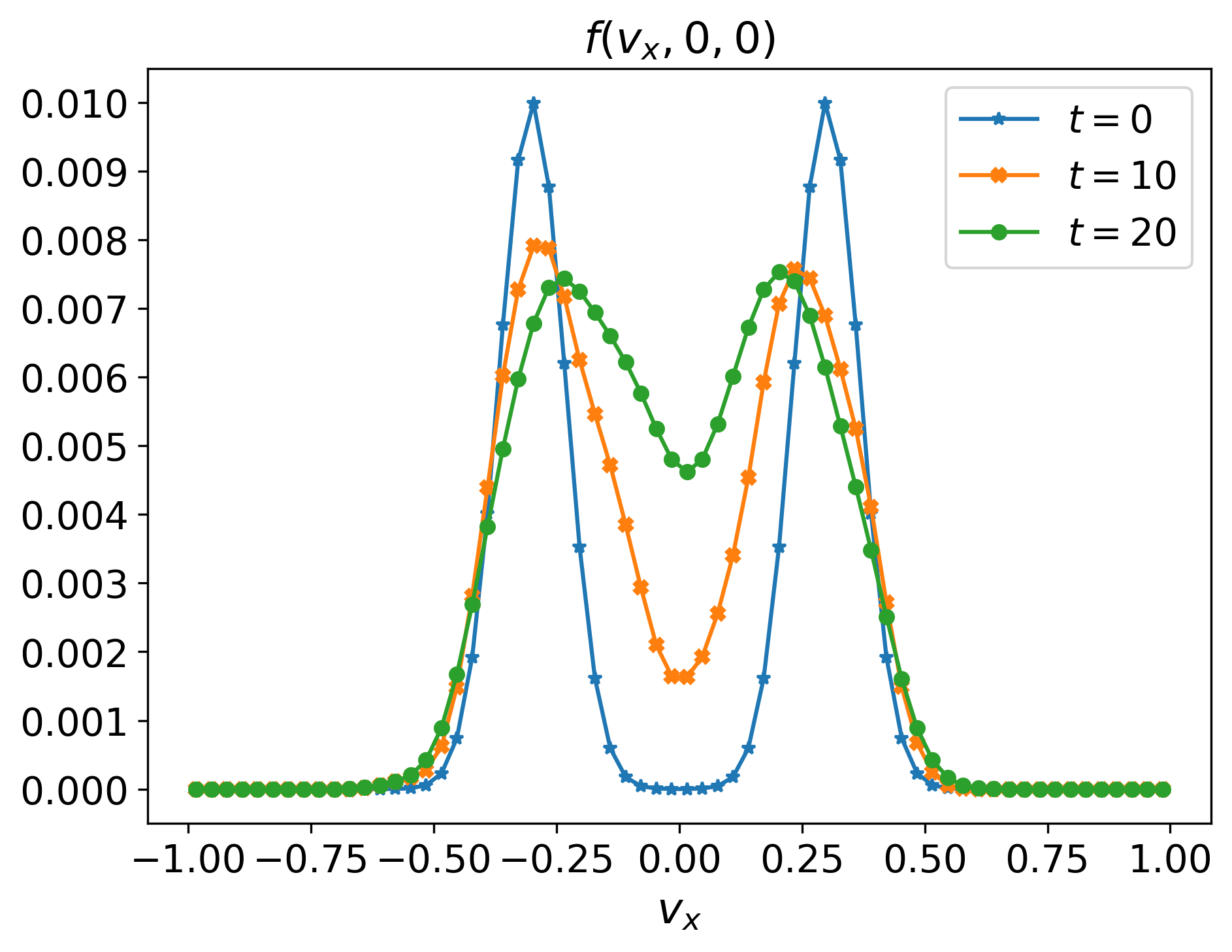}}
    \caption{Slices of the reconstructed solutions at $t=0$, $10$, and $20$.}
    \label{fig_eg4_kde}
\end{figure}
\\
\\
\noindent \textit{Efficiency}. \quad
To demonstrate the efficiency improvement in our approach, we compare the computation time for obtaining the score on GPU. To ensure a fair comparison, all the codes are written in PyTorch and executed on the Minnesota Supercomputer Institute with one Nvidia A40 GPU. 
As shown in Fig \ref{fig_eg4_time},  we observe a $\mathcal{O}(N)$ scaling of computational time in the score-based particle method. In contrast, the blob method \cite{carrillo2020particle} exhibits a computational time scaling of approximately $\mathcal{O}(N^{2})$.

Nevertheless, we would like to point out that even though the score-based particle method speeds up score evaluation, the summation in $N$ on the right-hand side of \eqref{vup} can be computationally expensive due to direct summations. To mitigate this, one could implement a treecode solver as demonstrated in \cite{carrillo2020particle}, reducing the cost to $\mathcal{O}(N\log N)$, or adopt the random batch particle method proposed in \cite{randombatch}, which reduces the cost to $\mathcal{O}(N^2/R)$, where $R$ is the number of batches. Since this paper primarily focuses on promoting the score-based concept, we leave further speed enhancements for future investigation.

\begin{figure}[H]
    \centerline{\includegraphics[scale=0.55]{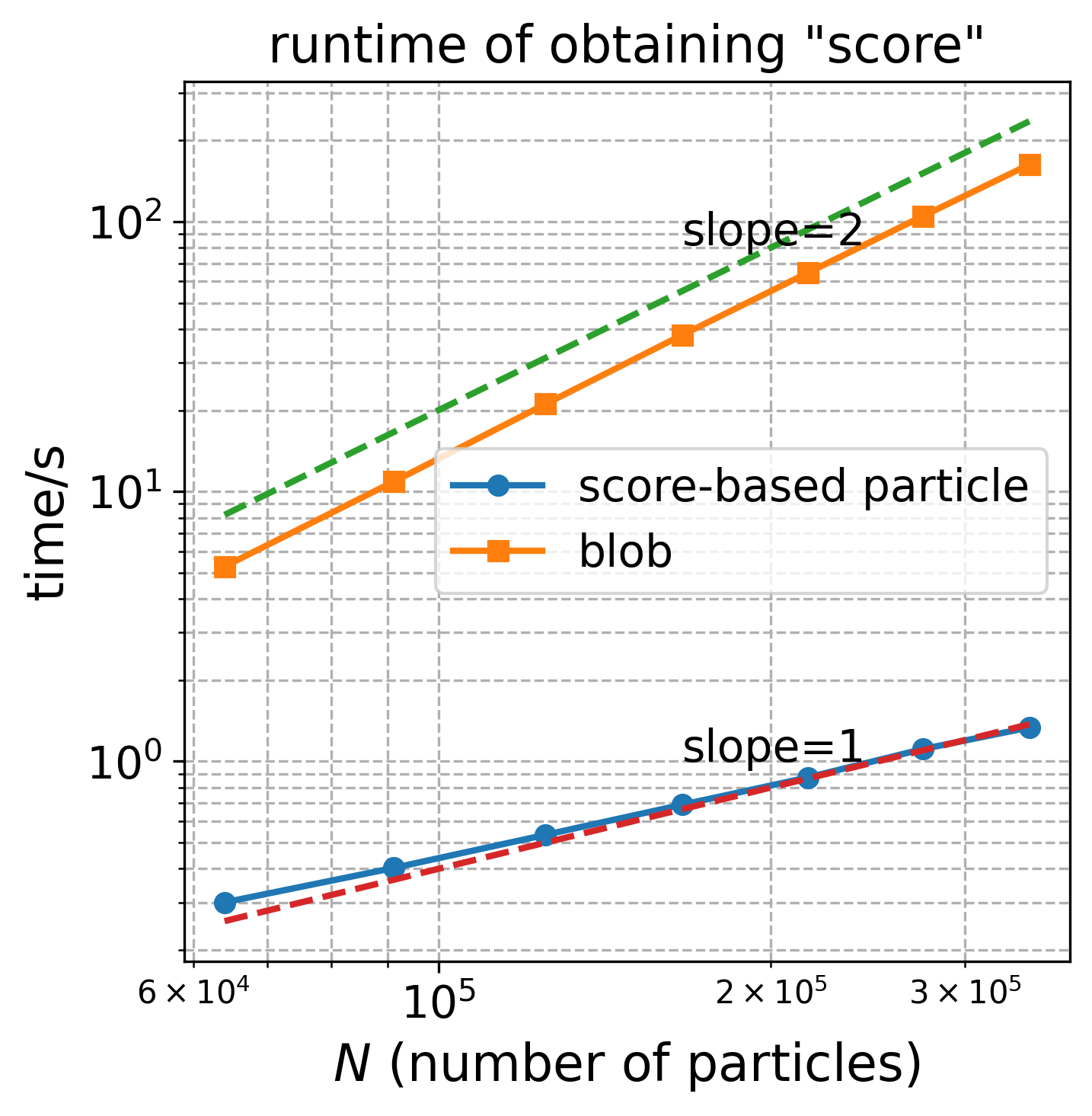}}
    \caption{GPU time (in seconds) for computing the score using the score-based particle method and the blob method with respect to particle number $N$.}
    \label{fig_eg4_time}
\end{figure}

\subsection{Density computation via Algorithm \ref{algorithm_densitiy}}
This subsection is dedicated to investigating the density computation outlined in Section \ref{sec:4}. We first examine the effectiveness of the formula (lines 5--6 in Algorithm~\ref{algorithm_densitiy}) when the score function is provided exactly. To do so, we revisit the example in Section \ref{sec5.1}. If the exact score is available, then the only expected errors are the Monte Carlo error which scales as $\mathcal{O}(N^{-\frac12})$ from initial sampling and the time discretization error $\mathcal{O}(\Delta t)$. In the following tests, we validate this order of accuracy by examining the numerical entropy at time $t^n$ defined by
\begin{equation*}
    \mathcal{H}_{\Delta t}^N(t_n) := \frac{1}{N} \sum_{i=1}^N \log f_{t_n}(\bv_i^n) \,.
\end{equation*}

To check the convergence in particle number $N$ and time step size $\Delta t$, we compute the following average $L^2$ error in $N$ and $L^1$ error in $\Delta t$:
\begin{equation*}
    e_N = \sqrt{\frac{1}{J} \sum_{j=1}^J \left| \mathcal{H}_{\Delta t}^N(t_{\operatorname{end}}) - \mathcal{H}_{\operatorname{ext}}(t_{\operatorname{end}}) \right|^2 } \,, 
    \quad \text{and} \quad e_{\Delta t} = \left| \mathcal{H}_{\Delta t}^N(t_{\operatorname{end}}) - \mathcal{H}_{\frac{\Delta t}{2}}^N(t_{\operatorname{end}}) \right| \,.
\end{equation*}
We compute $e_N$ over $J=20$ runs for each value of $N$, using different random seeds in each run to ensure independence. Fig. \ref{fig_5_N} shows the expected Monte Carlo convergence rate with respect to the particle number. We also observe the first-order accuracy in time in Fig. \ref{fig_5_t}.

\begin{figure}[H]
    \centering
    \subfloat[Monte Carlo rate with respect to particle number. We fix $t_{\operatorname{end}}=0.1$ and $\Delta t =0.01$. The particle number $N$ varies from $10^2$ to $10^4$.]
    {{\includegraphics[scale=0.45]{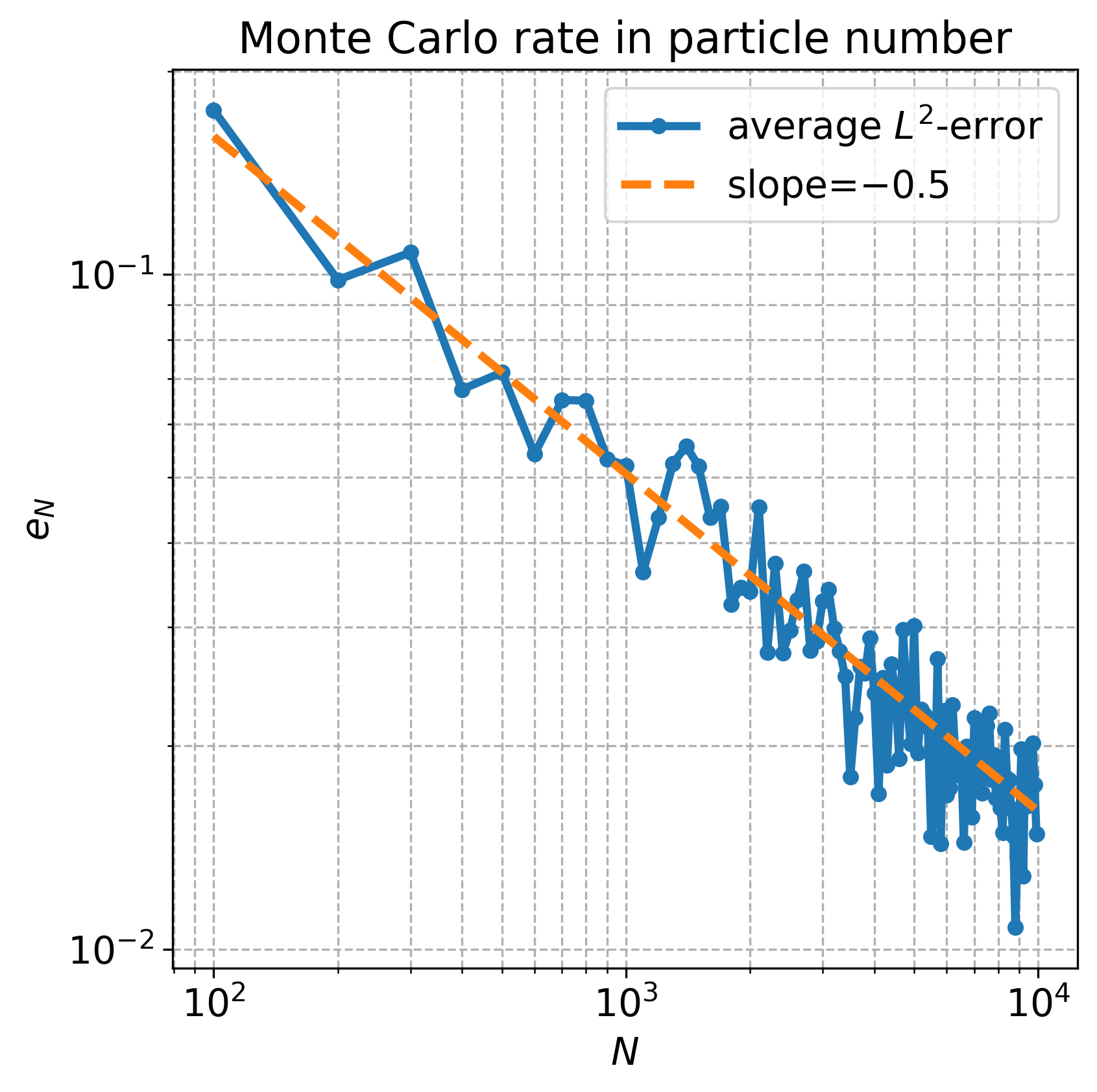}}\label{fig_5_N}}
    \hspace{4em}
    \subfloat[First order accuracy in time. We fix $t_{\operatorname{end}}=0.16$ and particle number $N=10^4$. The time step size $\Delta t=0.0025\,, 0.005\,, 0.01\,, 0.02\,, 0.04$.]{{\includegraphics[scale=0.45]{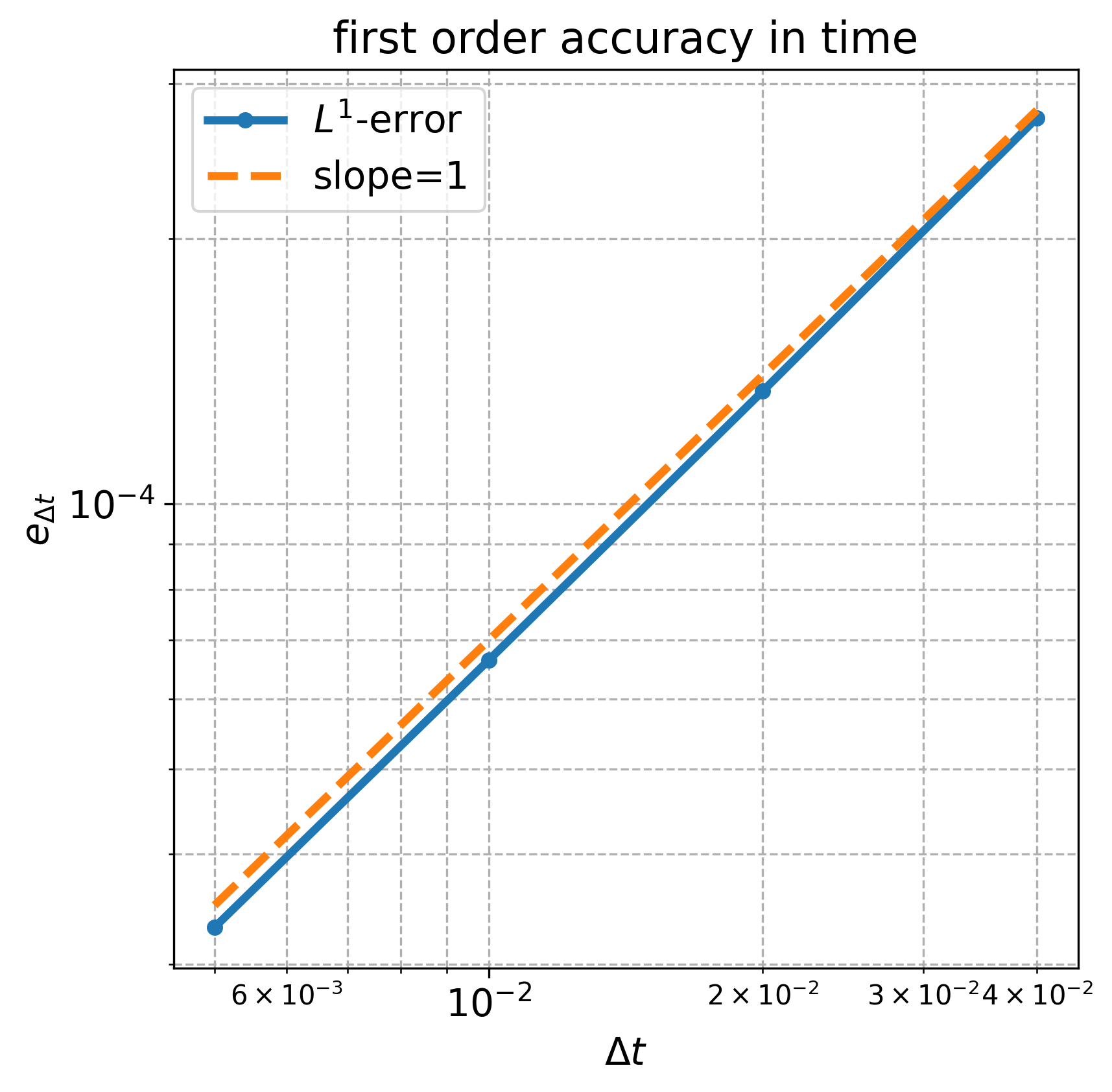}}\label{fig_5_t}}
    \caption{Convergence rate of Algorithm \ref{algorithm_densitiy} in $N$ and $t$.}
\end{figure}

Finally, we visually compare the numerical solution obtained by Algorithm \ref{algorithm_densitiy} using the score-matching trick with the analytical solution. The experimental setup is identical to that in Section \ref{sec5.1}. In Fig. \ref{fig_5_f_nn}, despite oscillations appearing near the top, the numerical solution still matches the analytical solution well. We attribute this to the fact that density computation uses the gradient of the score, while evolving particles only require the score itself, leading to higher accuracy demands on the score function for density computation.

To improve the accuracy of the learned score as well as its gradient, we propose the following parameterization of the score:
\begin{equation}\label{new_score}
    \bs(\bv) = h_{\theta}(|\bv|) \bv \,,
\end{equation}
where $h_{\theta}: \R \to \R$ is a neural network with the same architecture as in Section \ref{sec5.1}. This choice is motivated by the observation that any radial symmetric solution of the Landau equation---such as the BKW solution---results in the corresponding score taking the form of a product of a radially symmetric function with $\bv$. Fig. \ref{fig_5_f_new_nn} demonstrates that the new parametrization \eqref{new_score} reduces oscillations, resulting in a numerical solution that closely overlaps with the analytical solution.

\begin{figure}
    \centering
    \subfloat[Approximating score using a neural network directly.]{{\includegraphics[scale=0.45]{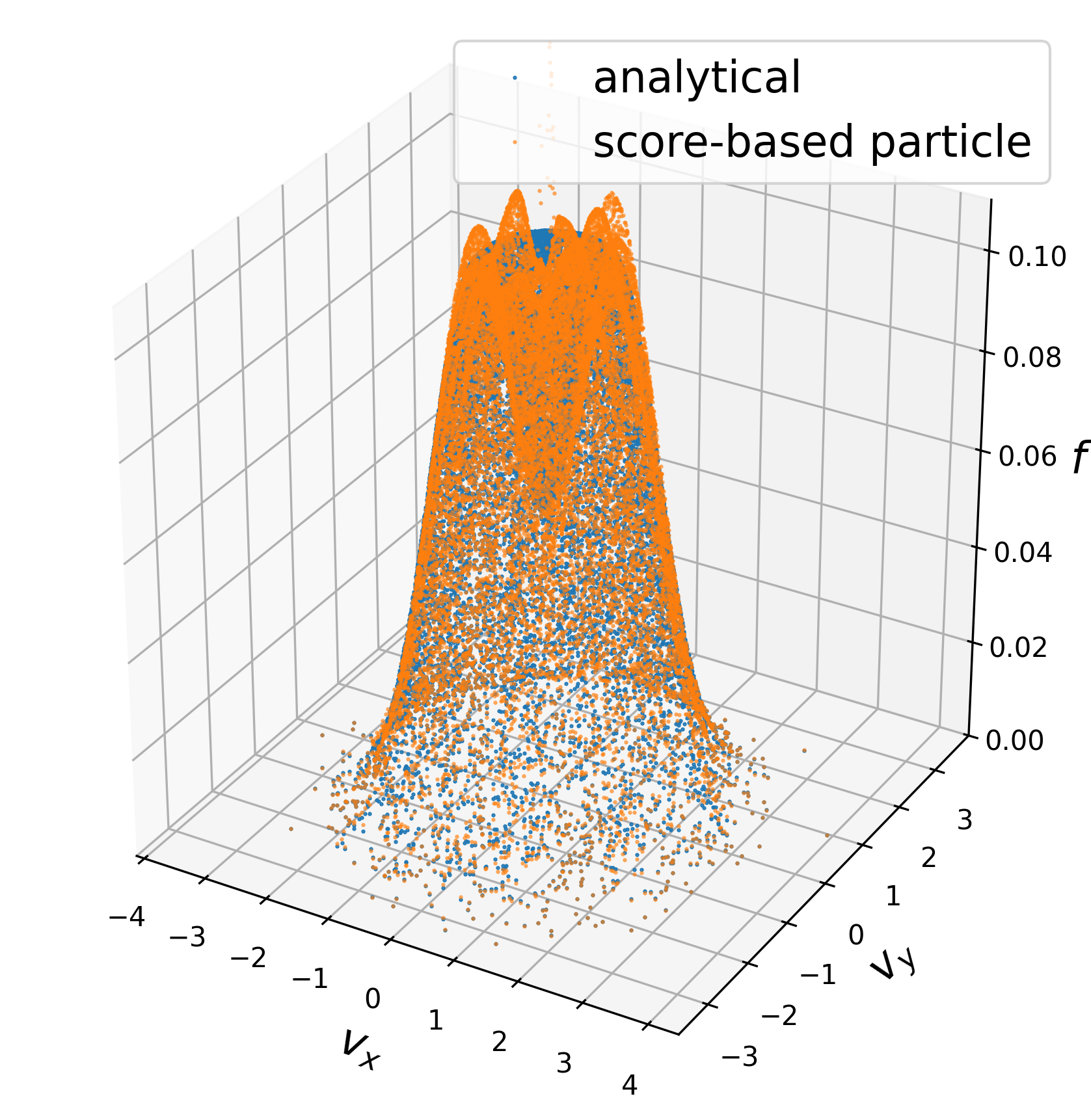}}\label{fig_5_f_nn}}
    \hspace{4em}
    \subfloat[Approximating score by new architecture \eqref{new_score}.]{{\includegraphics[scale=0.45]{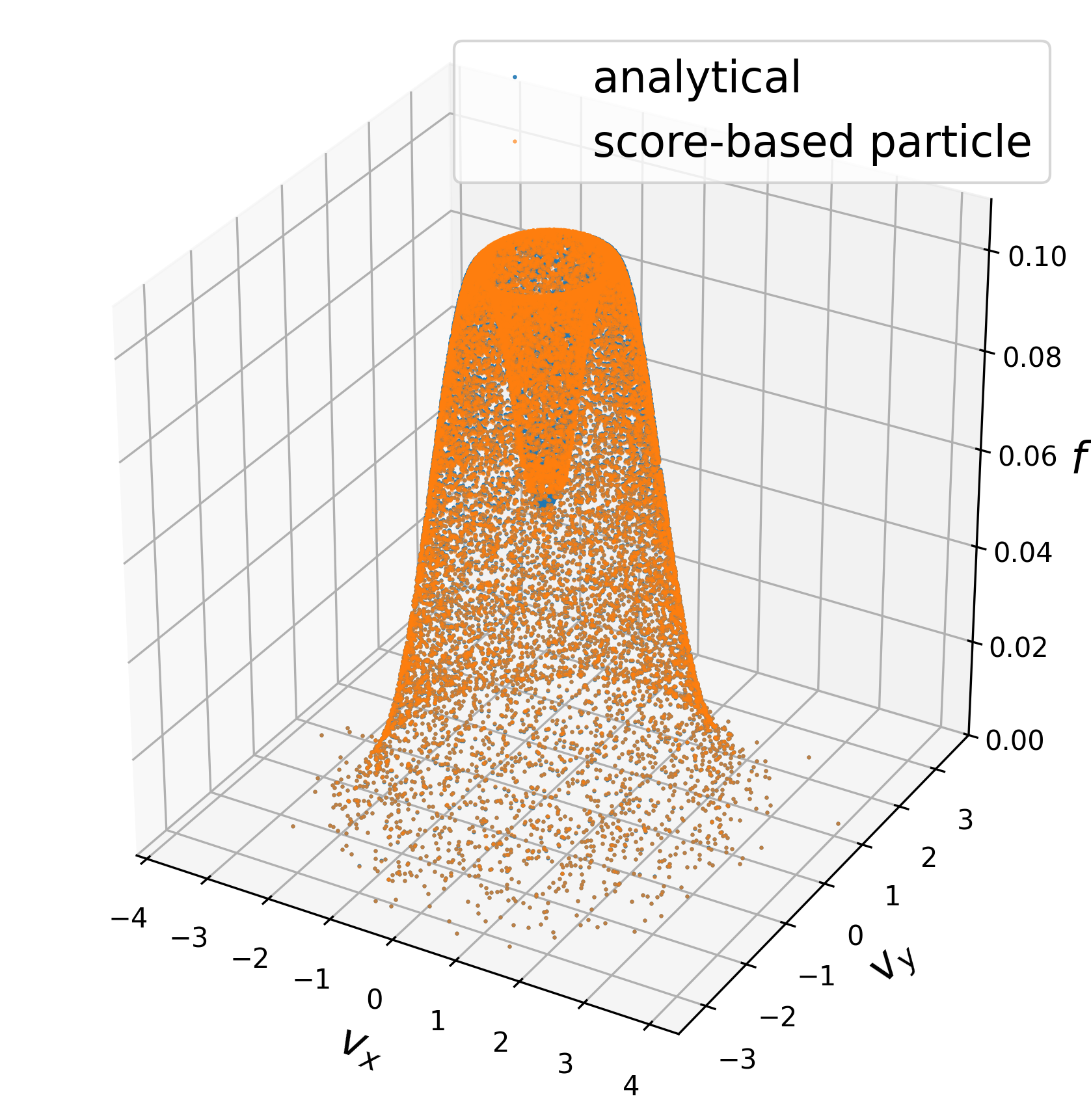}}\label{fig_5_f_new_nn}}
    \caption{Comparison of the analytical and numerical solutions obtained by Algorithm \ref{algorithm_densitiy} at $t=1$ using different score parametrization strategies.}
\end{figure}

\section{Conclusion} \label{sec:6}
In this paper, we introduce a structure-preserving score-based particle method for solving the space homogeneous Landau equation in plasmas. Our approach is rooted in interpreting the Landau equation as a continuity equation, enabling a deterministic particle method, which has been previously adopted in \cite{carrillo2020particle}. A key challenge in this approach is the nonlinear dependence of the velocity field on the density, which necessitates kernel density estimation. Our main contribution lies in recognizing that this nonlinearity takes the form of the score function, which can be efficiently learned from data using score-matching techniques. Additionally, we develop an algorithm for exact density computation from the flow map perspective, allowing direct calculation of quantities such as entropy. Finally, we provide a theoretical analysis demonstrating that the KL divergence between our approximation and the true solution can be effectively controlled using the score-matching loss. To make this method truly applicable to plasma simulations, it is natural to extend it to the spatially inhomogeneous case by combining it with the particle-in-cell method \cite{bailo2024collisional, chen2011energy}, and we plan to explore this extension in future work. Acceleration based on the random batch method \cite{randombatch} will also be investigated.

\appendix

\section{Proof of Proposition \ref{prop:pd}}\label{proof:pd}
\begin{proof}
    Since the integration domain is a torus, it suffices to show that $\int_{\bbT^3} A(\bv) \rho(\bv) \rd\bv \succ 0$. 
    \begin{equation*}
        \int_{\bbT^3} A(\bv) \rho(\bv) \rd\bv = \int_{\bbT^3} |\bv|^{\gamma}
        \begin{bmatrix}
            v_2^2+v_3^2 & -v_1 v_2 & -v_1 v_3 \\
            -v_1 v_2 & v_1^2+v_3^2 & -v_2 v_3 \\
            -v_1 v_3 & -v_2 v_3 & v_1^2+v_2^2 
        \end{bmatrix}
        \rho \rd\bv := B \,.
    \end{equation*}
    To show that the symmetric matrix $B$ is positive-definite, we only need to show that all the leading principal minors of $B$ are positive. In fact, the first-order leading principal minor is obviously positive. The second-order leading principal minor is given by
    \begin{flalign*}
            & \int_{\bbT^3} |\bv|^\gamma (v_2^2+v_3^2) \rho\rd\bv  
            \int_{\bbT^3} |\bv|^\gamma (v_1^2+v_3^2) \rho\rd\bv - 
            \left(\int_{\bbT^3} |\bv|^\gamma v_1 v_2 \rho\rd\bv \right)^2 \\
            = & \int_{\bbT^3} |\bv|^\gamma (|\bv|^2 -v_1^2) \rho\rd\bv  
            \int_{\bbT^3} |\bv|^\gamma (|\bv|^2 -v_2^2) \rho\rd\bv - 
            \left(\int_{\bbT^3} |\bv|^\gamma v_1 v_2 \rho\rd\bv \right)^2 \\
            = & \left(\int_{\bbT^3} |\bv|^{\gamma+2} \rho\rd\bv - \int_{\bbT^3} |\bv|^\gamma v_1^2 \rho\rd\bv \right)  
            \left(\int_{\bbT^3} |\bv|^{\gamma+2} \rho\rd\bv - \int_{\bbT^3} |\bv|^\gamma v_2^2 \rho \rd\bv \right) - 
            \left(\int_{\bbT^3} |\bv|^\gamma v_1 v_2 \rho\rd\bv \right)^2 \\
            = & \left(\int_{\bbT^3} |\bv|^{\gamma+2} \rho\rd\bv \right)^2 - \int_{\bbT^3} |\bv|^\gamma (v_1^2+v_2^2) \rho \rd\bv  \int_{\bbT^3} |\bv|^{\gamma+2} \rho\rd\bv 
            + \int_{\bbT^3} |\bv|^\gamma v_1^2 \rho\rd\bv \int_{\bbT^3} |\bv|^\gamma v_2^2 \rho\rd\bv - \left(\int_{\bbT^3} |\bv|^\gamma v_1 v_2 \rho\rd\bv \right)^2 \\
            \geq & \left(\int_{\bbT^3} |\bv|^{\gamma+2} \rho\rd\bv \right)^2 - \int_{\bbT^3} |\bv|^\gamma (v_1^2+v_2^2) \rho\rd\bv  \int_{\bbT^3} |\bv|^{\gamma+2} \rho\rd\bv \\
            = & \int_{\bbT^3} |\bv|^{\gamma+2} \rho\rd\bv  \int_{\bbT^3} |\bv|^\gamma v_3^2 \rho\rd\bv > 0 \,.
    \end{flalign*}
    The third-order leading principal minor (determinant) is given by $d_1 + d_2 + d_3$, where 
    \begin{flalign*}
        d_1 &= \int_{\bbT^3} |\bv|^\gamma (v_2^2+v_3^2) \rho\rd\bv 
        \left[ \int_{\bbT^3} |\bv|^\gamma (v_1^2+v_3^2) \rho\rd\bv \int_{\bbT^3} |\bv|^\gamma (v_1^2+v_2^2) \rho\rd\bv - \left(\int_{\bbT^3} |\bv|^\gamma v_2v_3 \rho\rd\bv \right)^2 \right] \\
        &= \int_{\bbT^3} |\bv|^\gamma (v_2^2+v_3^2) \rho\rd\bv 
        \bigg[ \left(\int_{\bbT^3} |\bv|^\gamma v_1^2 \rho\rd\bv\right)^2 
        + \int_{\bbT^3} |\bv|^\gamma v_1^2 \rho\rd\bv \int_{\bbT^3} |\bv|^\gamma v_2^2 \rho\rd\bv 
        + \int_{\bbT^3} |\bv|^\gamma v_3^2 \rho\rd\bv \int_{\bbT^3} |\bv|^\gamma v_1^2 \rho\rd\bv + \\
        & \int_{\bbT^3} |\bv|^\gamma v_3^2 \rho\rd\bv \int_{\bbT^3} |\bv|^\gamma v_2^2 \rho\rd\bv - \left(\int_{\bbT^3} |\bv|^\gamma v_2v_3 \rho\rd\bv \right)^2 \bigg] \\
        & \geq \int_{\bbT^3} |\bv|^\gamma (v_2^2+v_3^2) \rho\rd\bv 
        \left[ \left(\int_{\bbT^3} |\bv|^\gamma v_1^2 \rho\rd\bv\right)^2 
        + \int_{\bbT^3} |\bv|^\gamma v_1^2 \rho\rd\bv \int_{\bbT^3} |\bv|^\gamma v_2^2 \rho\rd\bv 
        + \int_{\bbT^3} |\bv|^\gamma v_3^2 \rho\rd\bv \int_{\bbT^3} |\bv|^\gamma v_1^2 \rho\rd\bv \right]
    \end{flalign*}
    \begin{flalign*}
        d_2 &= \int_{\bbT^3} |\bv|^\gamma v_1v_2 \rho\rd\bv
        \left[ -\int_{\bbT^3} |\bv|^\gamma v_1v_2 \rho\rd\bv \int_{\bbT^3} |\bv|^\gamma (v_1^2+v_2^2) \rho\rd\bv - \int_{\bbT^3} |\bv|^\gamma v_2v_3 \rho\rd\bv \int_{\bbT^3} |\bv|^\gamma v_1v_3 \rho\rd\bv \right] \\
        &= -\left( \int_{\bbT^3} |\bv|^\gamma v_1v_2 \rho\rd\bv \right)^2 \int_{\bbT^3} |\bv|^\gamma (v_1^2+v_2^2) \rho\rd\bv 
        - \int_{\bbT^3} |\bv|^\gamma v_1v_2 \rho\rd\bv \int_{\bbT^3} |\bv|^\gamma v_2v_3 \rho\rd\bv \int_{\bbT^3} |\bv|^\gamma v_1v_3 \rho\rd\bv
    \end{flalign*}
    \begin{flalign*}
        d_3 &= -\int_{\bbT^3} |\bv|^\gamma v_1v_3 \rho\rd\bv
        \left[ \int_{\bbT^3} |\bv|^\gamma v_1v_2 \rho\rd\bv \int_{\bbT^3} |\bv|^\gamma v_2v_3 \rho\rd\bv + \int_{\bbT^3} |\bv|^\gamma (v_1^2+v_3^2) \rho\rd\bv \int_{\bbT^3} |\bv|^\gamma v_1v_3 \rho\rd\bv \right] \\
        &= - \left( \int_{\bbT^3} |\bv|^\gamma v_1v_3 \rho\rd\bv \right)^2 \int_{\bbT^3} |\bv|^\gamma (v_1^2+v_3^2) \rho\rd\bv 
        - \int_{\bbT^3} |\bv|^\gamma v_1v_2 \rho\rd\bv \int_{\bbT^3} |\bv|^\gamma v_2v_3 \rho\rd\bv \int_{\bbT^3} |\bv|^\gamma v_1v_3 \rho\rd\bv
    \end{flalign*}
    Note that 
    \begin{equation*}
        \int_{\bbT^3} |\bv|^\gamma v_1^2 \rho\rd\bv \int_{\bbT^3} |\bv|^\gamma v_2^2 \rho\rd\bv \int_{\bbT^3} |\bv|^\gamma v_3^2 \rho\rd\bv
        - \int_{\bbT^3} |\bv|^\gamma v_1v_2 \rho\rd\bv \int_{\bbT^3} |\bv|^\gamma v_2v_3 \rho\rd\bv \int_{\bbT^3} |\bv|^\gamma v_1v_3 \rho\rd\bv > 0
    \end{equation*}
    \begin{equation*}
        \begin{split}
            & \int_{\bbT^3} |\bv|^\gamma v_2^2 \rho\rd\bv 
            \left[ \left(\int_{\bbT^3} |\bv|^\gamma v_1^2 \rho\rd\bv\right)^2 + \int_{\bbT^3} |\bv|^\gamma v_1^2 \rho\rd\bv \int_{\bbT^3} |\bv|^\gamma v_2^2 \rho\rd\bv \right]
            -\left( \int_{\bbT^3} |\bv|^\gamma v_1v_2 \rho\rd\bv \right)^2 \int_{\bbT^3} |\bv|^\gamma (v_1^2+v_2^2) \rho\rd\bv \\
            = & \left[\int_{\bbT^3} |\bv|^\gamma v_1^2 \rho\rd\bv \int_{\bbT^3} |\bv|^\gamma v_2^2 \rho\rd\bv - \left( \int_{\bbT^3} |\bv|^\gamma v_1v_2 \rho\rd\bv \right)^2\right]
            \int_{\bbT^3} |\bv|^\gamma (v_1^2+v_2^2) \rho\rd\bv
            > 0
        \end{split}
    \end{equation*}
    \begin{equation*}
        \begin{split}
            & \int_{\bbT^3} |\bv|^\gamma v_3^2 \rho\rd\bv 
            \left[ \left(\int_{\bbT^3} |\bv|^\gamma v_1^2 \rho\rd\bv\right)^2 + \int_{\bbT^3} |\bv|^\gamma v_3^2 \rho\rd\bv \int_{\bbT^3} |\bv|^\gamma v_1^2 \rho\rd\bv \right]
            -\left( \int_{\bbT^3} |\bv|^\gamma v_1v_3 \rho\rd\bv \right)^2 \int_{\bbT^3} |\bv|^\gamma (v_1^2+v_3^2) \rho\rd\bv \\
            = & \left[\int_{\bbT^3} |\bv|^\gamma v_1^2 \rho\rd\bv \int_{\bbT^3} |\bv|^\gamma v_3^2 \rho\rd\bv - \left( \int_{\bbT^3} |\bv|^\gamma v_1v_3 \rho\rd\bv \right)^2\right]
            \int_{\bbT^3} |\bv|^\gamma (v_1^2+v_3^2) \rho\rd\bv
            > 0
        \end{split}
    \end{equation*}
    Therefore, $d_1+d_2+d_3 > 0$.
\end{proof}

\section{Proof of Corollary~\ref{logdet_coro}}\label{app_proof}

\begin{lemma}\label{div_matrix}
If $A(\bu) \in \mathbb{R}^{d \times d}$, $\boldsymbol{x}(\bu) \in \mathbb{R}^d$, and $\bu \in \mathbb{R}^{d}$, then
\begin{equation*}
    \nabla_{\bu} \cdot (A(\bu)\boldsymbol{x}(\bu)) = (\nabla_{\bu} \cdot A(\bu)^\top) \cdot \boldsymbol{x}(\bu) + A(\bu) : \nabla_{\bu} \boldsymbol{x}(\bu)^\top,
\end{equation*}
where $A:B := \sum_{ij}A_{ij}B_{ij}$ and the divergence of matrix is applied row-wise. 
\end{lemma}
\begin{proof}
\begin{equation*}
    \begin{split}
        \nabla_{\bu} \cdot (A(\bu)\boldsymbol{x}(\bu)) 
        & = \frac{\rd(a_{11}x_1+\cdots+a_{1d}x_d)}{\rd u_1} + \cdots + \frac{\rd(a_{d1}x_1+\cdots+a_{dd}x_d)}{\rd u_d} \\
        & = x_1\frac{\rd a_{11}}{\rd u_1} + \cdots + x_d\frac{\rd a_{1d}}{\rd u_1} + \cdots + x_1\frac{\rd a_{d1}}{\rd u_d} + \cdots + x_d\frac{\rd a_{dd}}{\rd u_d} + \\ 
        & \quad + a_{11} \frac{\rd x_1}{\rd u_1} + \cdots + a_{1d} \frac{\rd x_d}{\rd u_1} + \cdots + a_{d1} \frac{\rd x_1}{\rd u_d} + \cdots + a_{dd} \frac{\rd x_d}{\rd u_d} \\
        & = (\nabla_{\bu} \cdot A(\bu)^\top) \cdot \boldsymbol{x}(\bu) + A(\bu) : \nabla_{\bu} \boldsymbol{x}(\bu)^\top.
    \end{split}
\end{equation*}
\end{proof}

Now we present the proof for Corollary \ref{logdet_coro}.
\begin{proof}
By Lemma \ref{div_matrix},
\begin{flalign*}
    & \nabla \cdot \left[\Pi(\bT(\bV,t)-\bT(\bV_*,t)) \left(\tbs_t(\bT(\bV,t)) - \tbs_t(\bT(\bV_*,t))\right)\right] \\
    = & [\nabla \cdot \Pi(\bT(\bV,t)-\bT(\bV_*,t))^\top] \cdot (\tbs_t(\bT(\bV,t)) - \tbs_t(\bT(\bV_*,t))) + \Pi(\bT(\bV,t)-\bT(\bV_*,t)) : \nabla \tbs_t(\bT(\bV,t))^\top \\
    = & [\nabla \cdot \Pi(\bT(\bV,t)-\bT(\bV_*,t))] \cdot (\tbs_t(\bT(\bV,t)) - \tbs_t(\bT(\bV_*,t))) + \Pi(\bT(\bV,t)-\bT(\bV_*,t)) : \nabla \tbs_t(\bT(\bV,t))^\top.
\end{flalign*}
Let $\bz := \bT(\bV,t)-\bT(\bV_*,t)$. Then
\begin{equation*}
    \Pi(\bz)_{ij} = 
    \begin{cases}{}
        -\frac{z_i z_j}{|\bz|^2}, \quad \text{for } i\not=j \,, \\
        1 - \frac{z_i^2}{|\bz|^2}, \quad \text{for } i=j \,,
    \end{cases}
    \implies \nabla \cdot \Pi(\bz) = -(d-1)\frac{\bz}{|\bz|^2} \,.
\end{equation*}

Therefore, the evolution for $\log |\det \nabla_\bV \bT(\bV,t)|$ is
\begin{flalign*}
    \frac{\rd}{\rd t} \log |\det \nabla_\bV \bT(\bV,t)| 
    & = - \int_{\mathbb{R}^{d}} \big\{ A(\bT(\bV,t)-\bT(\bV_*,t)) : \nabla \tbs_t(\bT(\bV,t))^\top -C_{\gamma}(d-1) \\
    & |\bT(\bV,t)-\bT(\bV_*,t)|^{\gamma} (\bT(\bV,t)-\bT(\bV_*,t)) \cdot \left(\tbs_t(\bT(\bV,t)) - \tbs_t(\bT(\bV_*,t))\right) \big\} f_0(\bV_*) \rd\bV_*.
\end{flalign*}
\end{proof}

\bibliographystyle{siam}
\bibliography{ref.bib}

\end{document}